\documentclass[11pt,letterpaper]{article}
\usepackage[T1]{fontenc}
\usepackage[utf8]{inputenc}
\usepackage{lmodern}
\usepackage[margin=1in]{geometry}
\usepackage{amsmath,amssymb,amsthm,mathtools}
\usepackage{mathrsfs}
\usepackage{microtype}
\usepackage{enumitem}
\usepackage[hidelinks]{hyperref}
\usepackage{bookmark}
\numberwithin{equation}{section}
\newtheorem{theorem}{Theorem}[section]
\newtheorem{lemma}[theorem]{Lemma}

\theoremstyle{definition}

\theoremstyle{remark}

\newcommand{\R}{\mathbb R}
\newcommand{\Sn}{S^{n-1}}
\newcommand{\E}{\mathbb E}
\newcommand{\cK}{\mathcal K^n_0}
\newcommand{\cH}{\mathcal H}
\newcommand{\cP}{\mathcal P}
\newcommand{\cD}{\mathcal D}
\newcommand{\vbar}{\overline V}
\newcommand{\ind}{\mathbf 1}
\newcommand{\HS}{\mathrm{HS}}
\newcommand{\loc}{\mathrm{loc}}
\DeclareMathOperator{\Var}{Var}
\DeclareMathOperator{\Cov}{Cov}
\DeclareMathOperator{\tr}{tr}
\DeclareMathOperator{\supp}{supp}
\DeclareMathOperator{\Span}{span}
\DeclareMathOperator{\Sym}{Sym}
\DeclareMathOperator{\diver}{div}

\DeclareMathOperator{\argmin}{argmin}
\DeclareMathOperator{\intt}{int}
\setlist[enumerate]{label=\textup{(\roman*)},leftmargin=2em,itemsep=2pt,topsep=4pt}
\allowdisplaybreaks[2]
\hypersetup{pdftitle={The centered logarithmic Brunn--Minkowski inequality},pdfauthor={Yude Liu, Qiang Sun, Ge Xiong},pdfsubject={Centered logarithmic Brunn--Minkowski and Minkowski inequalities; cone-volume and Lp surface area measures}}

\title{The centered logarithmic Brunn--Minkowski inequality}
\author{Yude Liu\thanks{School of Mathematical Sciences, Fudan University, Shanghai 200433, China.\newline Email: \texttt{liuyude@fudan.edu.cn} (Y.~Liu).}
\and Qiang Sun\thanks{Shanghai Center for Mathematical Sciences, Fudan University, Shanghai 200438, China.\newline Email: \texttt{sunqiang@fudan.edu.cn} (Q.~Sun).}
\and Ge Xiong\thanks{School of Mathematical Sciences, Tongji University, Shanghai 200092, China.\newline Email: \texttt{xiongge@tongji.edu.cn} (G.~Xiong).}}
\date{}
\begin{document}
\maketitle
\begin{abstract}
The logarithmic Brunn--Minkowski and logarithmic Minkowski inequalities are proved for convex bodies whose centroids are at the origin, with equality precisely for independent positive dilations of common direct summands. This characterizes when centered convex bodies have the same cone-volume measure. The \(L_p\) Brunn--Minkowski and Minkowski inequalities are also obtained for every \(p>0\), with equality only for positive dilates, and the corresponding uniqueness of centered \(L_p\) surface area measures is established.
\end{abstract}

\section{Introduction}\label{sec:introduction}\

A convex body in $\R^n$ is a compact convex set with nonempty interior. Write $|K|$ for its volume. We call $K$ \emph{centered} if its centroid
\[
 g(K)=\frac1{|K|}\int_K x\,dx
\]
is at the origin.  The support function of $K$ is
\[
 h_K(u)=\max_{x\in K}\langle x,u\rangle,\qquad u\in\Sn.
\]
For a positive continuous function $h$ on $\Sn$, its Wulff body is
\[
 W[h]=\bigcap_{u\in\Sn}\{x\in\R^n:\langle x,u\rangle\le h(u)\}.
\]

The first result is the centered logarithmic Brunn--Minkowski inequality, including the equality cases formulated in \cite[Conjecture 3.4]{BoroczkySurvey}.

\begin{theorem}\label{thm:logbm}
Let $n\ge2$ and $K,L\subset\R^n$ be centered convex bodies. Then for each $t\in(0,1)$,
\begin{equation}\label{eq:main-bm}
 \big|W[h_K^{1-t}h_L^t]\big|\ge |K|^{1-t}|L|^t
\end{equation}
with equality if and only if there are nonzero complementary linear subspaces $E_1,\ldots,E_r$, centered convex bodies $K_i\subset E_i$ relative to these subspaces, and positive numbers $a_i$ such that
\begin{equation}\label{eq:main-split}
 \R^n=E_1\oplus\cdots\oplus E_r,\qquad
 K=K_1\oplus\cdots\oplus K_r,\qquad
 L=a_1K_1\oplus\cdots\oplus a_rK_r.
\end{equation}
\end{theorem}

Here the direct sum of convex sets is their Minkowski sum in complementary subspaces. The subspaces need not be orthogonal, and the factors need not be origin-symmetric. The case $r=1$ gives positive dilates. In dimension one every centered interval is origin-symmetric and the corresponding inequality is an equality.

The first variation of volume leads to the logarithmic Minkowski inequality. Let $S_K$ be the surface area measure of $K$. Its cone-volume measure and its probability normalization are
\begin{equation*}
 dV_K=\frac1n h_K\,dS_K,\qquad V_K(\Sn)=|K|,\qquad
 \vbar_K=\frac{V_K}{|K|}.
\end{equation*}

\begin{theorem}\label{thm:logm}
Let $n\ge2$ and let $K,L\subset\R^n$ be centered convex bodies. Then
\begin{equation}\label{eq:main-m}
 \int_{\Sn}\log\frac{h_L}{h_K}\,dV_K
 \ge \frac{|K|}{n}\log\frac{|L|}{|K|}
\end{equation}
with equality if and only if $K,L$ admit the common direct-sum representation \eqref{eq:main-split}.
\end{theorem}

The equality statement characterizes centered convex bodies with the same normalized cone-volume measure.

\begin{theorem}\label{thm:uniqueness}
Let $n\ge2$ and let $K,L\subset\R^n$ be centered convex bodies. Then $
 \vbar_K=\vbar_L,
$
if and only if they admit \eqref{eq:main-split}.
\end{theorem}

The positive-parameter conclusions use the notation of the $L_p$ Brunn--Minkowski theory \cite{Firey62,Lutwak93}. For $p>0$ and $0\le t\le1$, put
\begin{equation*}
 C_{p,t}(K,L)=W\!\left[((1-t)h_K^p+th_L^p)^{1/p}\right].
\end{equation*}
The $L_p$ surface area measure and mixed volume are
\begin{equation}\label{eq:lp-standard}
 dS_p(K,u)=h_K(u)^{1-p}\,dS_K(u),\qquad
 V_p(K,L)=\frac1n\int_{\Sn}h_L^p\,dS_p(K).
\end{equation}
For a convex body $K$ containing the origin in its interior, define
\begin{equation*}
 \Psi(K)=\min_{\xi\in\R^n}\int_K e^{-\langle\xi,x\rangle}\,dx,
 \qquad
 \xi_K=\underset{\xi\in\R^n}{\operatorname{argmin}}
              \int_K e^{-\langle\xi,x\rangle}\,dx.
\end{equation*}
The minimizer exists uniquely, and $\Psi(K)=|K|$ exactly when $K$ is centered; these facts are proved in Section~\ref{sec:preliminaries}. We also use the optimized weighted mixed integral
\begin{equation}\label{eq:lp-weighted-Vp}
 \widetilde V_p(K,L)=\frac1n\int_{\partial K}
 h_L(\nu_K)^p h_K(\nu_K)^{1-p}e^{-\langle\xi_K,x\rangle}
 \,d\mathcal H^{n-1}(x).
\end{equation}
The weight is determined by the first body. If $K$ is centered, then $\xi_K=0$ and $\widetilde V_p(K,L)=V_p(K,L)$.

\begin{theorem}\label{thm:lp-bm}
Let $n\ge2$ and $0<p<\infty$. For convex bodies $K,L$ containing the origin in their interiors and every $t\in[0,1]$,
\begin{equation}\label{eq:lp-weighted-BM}
 \Psi(C_{p,t}(K,L))^{p/n}
 \ge(1-t)\Psi(K)^{p/n}+t\Psi(L)^{p/n}.
\end{equation}
In particular, for centered convex bodies $K,L$,
\begin{equation}\label{eq:lp-centered-BM}
 |C_{p,t}(K,L)|^{p/n}
 \ge(1-t)|K|^{p/n}+t|L|^{p/n}.
\end{equation}
For each fixed $t\in(0,1)$, equality in either  inequality holds if and only if $L=cK$ for some $c>0$.
\end{theorem}

For centered $K,L$ and $0<t<1$, the weaker geometric-mean bound
\begin{equation}\label{eq:lp-centered-geometric}
 |C_{p,t}(K,L)|\ge |K|^{1-t}|L|^t
\end{equation}
has equality if and only if $K=L$. Thus the extra direct-sum equality cases in Theorem~\ref{thm:logbm} occur only at the logarithmic endpoint.

\begin{theorem}\label{thm:lp-minkowski}
Let $n\ge2$ and $0<p<\infty$. For convex bodies $K,L$ containing the origin in their interiors,
\begin{equation}\label{eq:lp-weighted-minkowski}
 \widetilde V_p(K,L)\ge\Psi(K)^{1-p/n}\Psi(L)^{p/n}.
\end{equation}
In particular, for centered convex bodies $K,L$,
\begin{equation}\label{eq:lp-centered-minkowski}
 V_p(K,L)\ge |K|^{1-p/n}|L|^{p/n}.
\end{equation}
Equality in either inequality holds if and only if $L=cK$ for some $c>0$. Consequently, for centered $K,L$,
\begin{equation}\label{eq:lp-surface-uniqueness}
 S_p(K,\cdot)=S_p(L,\cdot)
 \quad\Longleftrightarrow\quad
 \begin{cases}
 K=L,&p\ne n,\\
 L=cK\text{ for some }c>0,&p=n.
 \end{cases}
\end{equation}
\end{theorem}

The volume and mixed-volume inequalities for $p\ge1$ belong to the classical Brunn--Minkowski--Firey theory \cite{Firey62,Lutwak93,Schneider14}; at $p=1$, centering removes the translation freedom in the classical equality case. The conclusions for $0<p<1$ follow here for arbitrary centered pairs, without a symmetry assumption. The optimized inequalities are formulated on all bodies containing the origin and provide the common proof. The exception at $p=n$ in \eqref{eq:lp-surface-uniqueness} is forced by the homogeneity $S_p(cK,\cdot)=c^{n-p}S_p(K,\cdot)$, not by a loss of regularity.

The logarithmic Brunn--Minkowski conjecture was introduced by B\"or\"oczky, Lutwak, Yang and Zhang \cite{BLYZ12}. They proved the origin-symmetric planar case, including its equality cases, and related the volume inequality to the logarithmic Minkowski inequality. The centered formulations, with the common direct-sum equality structure, are recorded in \cite[Conjectures 3.2 and 3.4]{BoroczkySurvey}. The distinction between centering and symmetry is essential: a logarithmic Wulff combination of centered bodies need not be centered, and the odd translation modes cannot simply be discarded.

Several earlier results impose geometric symmetries or special structure. Saroglou proved the logarithmic inequality for unconditional pairs \cite{Saroglou15}. B\"or\"oczky and Kalantzopoulos treated pairs invariant under common systems of linear reflections whose fixed hyperplanes intersect only at the origin, and analyzed equality \cite{BK22}. Rotem obtained the complex-body case through complex interpolation \cite{Rotem14}. Van Handel established the local logarithmic inequality for origin-symmetric zonoids \cite{vH23}, while Tao, Xiong and Xiong studied the logarithmic Minkowski inequality and uniqueness for cylinders \cite{TXX23}. These hypotheses are distinct from the unrestricted centroid condition used here.

A complementary line of work studies the second variation of volume. Colesanti, Livshyts and Marsiglietti established a local logarithmic result near the Euclidean ball \cite{CLM17}. Kolesnikov and Milman developed the spectral formulation of local $L_p$ inequalities \cite{KM22}; Chen, Huang, Li and Liu obtained the global symmetric inequality for $p$ sufficiently close to one, and proved the equivalence of the local and global symmetric formulations \cite{CHLL20}; Putterman  gave another proof of  the equivalence \cite{Putterman21}. Milman's centro-affine approach gives further curvature-dependent and isomorphic logarithmic Minkowski inequalities \cite{Milman25}. More recently, Iffland treated bodies of revolution, allowing nonsymmetric comparison bodies under a weighted centering condition \cite{Iffland26}; this condition is not the ordinary centroid condition on the comparison body. Lu proved the global symmetric inequality under $n-2$ common orthogonal reflection symmetries \cite{Lu26}. Xi gave a new planar argument using dilation position \cite{Xi26}.

The connection with cone-volume measures also has an existence side. B\"or\"oczky, Lutwak, Yang and Zhang characterized the nonzero finite even measures realized by origin-symmetric convex bodies through the subspace concentration condition \cite{BLYZ13}. Henk and Linke established the subspace concentration inequality for centered polytopes \cite{HL14}, and B\"or\"oczky and Henk extended it to general centered convex bodies \cite{BH16}. These necessary concentration inequalities do not constitute an existence theorem for arbitrary noneven data. Theorem~\ref{thm:uniqueness} instead classifies the ambiguity between existing centered realizations; the direct-sum exceptions are consistent with the cylinder phenomena in \cite{TXX23}.

The analytic starting point is the moment-coordinate argument of the September 23, 2026 preprint \cite{OpenAI26}, which states the logarithmic Brunn--Minkowski inequality for arbitrary origin-symmetric convex bodies. Its compact adjoint identities and tensor calculation are reproduced below without symmetry assumptions. The retained-remainder approach to equality comes from the predecessor manuscript \cite{SunDraft}. The present proof replaces the even test space by the orthogonal complement of all affine functions and allows the limiting target to carry a nonconstant exponential density. Neither the symmetric volume theorem nor an equality classification is used as an unproved implication in the argument below.

The moment-measure results of Berman--Berndtsson and Cordero-Erausquin--Klartag supply the underlying coordinates \cite{BB13,CEK15}. Klartag's estimate bounds their Hessian from above \cite{Klartag13}, and Fathi identifies the target-coordinate Hessian as a Stein kernel \cite{Fathi19}. Related links between homogeneous variance inequalities and local $L_p$ inequalities were developed by Kolesnikov, Livshyts and Rotem \cite{KLR26}. The equality limit uses quantitative stability of moment measures \cite{BR26}, local convergence of log-concave densities \cite{CS10}, and interior Monge--Amp\`ere estimates \cite{FJM16}. The first variation of the optimized functional is obtained from the weighted Aleksandrov formula in \cite[Lemma 2.4]{FLMZ}.

The centered class need not be preserved by Wulff interpolation. To avoid applying an inequality outside its hypotheses, we work with
\begin{equation*}
 \cP(C)=\log\Psi(C),
 \qquad
 \cH(a)=\cP(W[e^a]),\quad a\in C(\Sn).
\end{equation*}
The minimizing exponential density has centroid zero. Moreover,
\[
 \cP(C)\le\log|C|,\qquad
 \cP(C)=\log|C|\quad\Longleftrightarrow\quad g(C)=0.
\]
We prove that $\cH$ is concave on the entire space $C(\Sn)$. Its values at centered endpoints are their logarithmic volumes, so this stronger functional assertion gives \eqref{eq:main-bm} even when an intermediate body is not centered.

For a finite family of one-sided support constraints, the regularization has potential
\[
 V(x)=\frac\varepsilon2|x|^2+\sum_i\psi_i(x)+\langle\xi,x\rangle,
 \qquad
 \psi_i(x)=\left(\frac{\langle p_i,x\rangle_+}{h_i}\right)^q.
\]
The tilt $\xi$ is chosen to center the resulting probability $\mu$. Optimizing in $\xi$ produces a covariance subtraction in the second derivative. The required bound is
\begin{equation*}
 \Var_\mu(F)-\E_\mu[xF]^T\big(\E_\mu[xx^T]\big)^{-1}\E_\mu[xF]
 \le\sum_i b_i^2\E_\mu\psi_i,
 \qquad F=\sum_i b_i\psi_i.
\end{equation*}
In moment coordinates the left side is the squared norm of the projection of $F$ away from affine functions. Compact double-divergence tests are dense in this orthogonal complement. A tensor identity with a sum-of-squares remainder proves this bound.

For equality, the optimized functional is affine along the logarithmic segment. We choose regularizations whose concavity deficit is of order $\varepsilon$ and retain the term $\varepsilon\E|W|^2$ in the remainder. This gives both a vanishing local derivative norm for a matrix field and a uniform global flux bound. The limiting matrix splits the moment potential and its gradient image. A common weighted cone measure transfers the splitting to the endpoints. The same affine-segment argument treats equality in Theorem~\ref{thm:logm}; Theorem~\ref{thm:uniqueness} then follows by applying that inequality in both orders.

Section~\ref{sec:preliminaries} collects the geometric and variational facts. Section~\ref{sec:proof-bm} proves Theorem~\ref{thm:logbm}, integrating the projected variance estimate and the rigidity argument into its proof. Section~\ref{sec:proof-m-uniqueness} proves Theorems~\ref{thm:logm} and~\ref{thm:uniqueness}. Section~\ref{sec:lp-proofs} proves Theorems~\ref{thm:lp-bm} and~\ref{thm:lp-minkowski}, including the optimized forms, the logarithmic limit, and the failure of the corresponding universal extensions for negative exponents. The positive-parameter equality arguments use functional concavity and strict scalar inequalities, independently of the logarithmic equality classification.

\section{Preliminaries}\label{sec:preliminaries}\ 

Let $\mathcal K^n_o$ denote the convex bodies containing the origin in their interiors, and let $\cK$ be the centered subclass. Write $B$ for the Euclidean unit ball and $\mathcal H^{n-1}$ for $(n-1)$-dimensional Hausdorff measure. Relative volumes and surface area measures are used for bodies in a proper subspace. The outer unit normal $\nu_K(x)$ exists at $\mathcal H^{n-1}$-almost every point of $\partial K$. For a Borel set $\omega\subset\Sn$,
\begin{equation*}
 V_K(\omega)=\frac1n\int_{\nu_K^{-1}(\omega)}
 \langle x,\nu_K(x)\rangle\,d\mathcal H^{n-1}(x).
\end{equation*}
In particular, $\supp V_K=\supp S_K$ for $K\in\mathcal K^n_o$. Standard facts about support functions, mixed volumes, and surface area measures can be found in \cite{Schneider14}.

The Wulff construction is monotone and satisfies $W[ch]=cW[h]$ for $c>0$. If $a,b\in C(\Sn)$, then
\begin{equation}\label{eq:wulff-sandwich}
 e^{-\|a-b\|_\infty}W[e^b]\subset W[e^a]
 \subset e^{\|a-b\|_\infty}W[e^b].
\end{equation}
For $T\in\mathrm{GL}(n)$ define
\[
 \vartheta_T(u)=\frac{T^{-T}u}{|T^{-T}u|}.
\]
The support transformation and the boundary formula give
\begin{equation}\label{eq:covariance}
 W[h_{TK}^{1-t}h_{TL}^t]=T W[h_K^{1-t}h_L^t],
 \qquad V_{TK}=|\det T|(\vartheta_T)_\#V_K.
\end{equation}
For the second identity, the surface Jacobian is $|\det T|\,|T^{-T}u|$, whereas the supporting distance is divided by $|T^{-T}u|$. Thus equality of normalized cone-volume measures is invariant under a common invertible linear map. Centering is preserved as well, since $g(TK)=Tg(K)$.

We use Aleksandrov's first-variation formula in the form
\begin{equation}\label{eq:aleksandrov}
 \left.\frac{d}{ds}\right|_{s=0}|W[h_K e^{sf}]|
 =\int_{\Sn}fh_K\,dS_K=n\int_{\Sn}f\,dV_K,
 \qquad f\in C(\Sn).
\end{equation}
The derivative is two-sided and requires no smoothness of $\partial K$; see \cite{BLYZ12,BLYZ13}. We also use its weighted version with the smooth positive density $e^{-\langle\xi,x\rangle}$, as in \cite[Lemma 2.4]{FLMZ}.

\begin{lemma}\label{lem:reconstruct}
If $Q\in\mathcal K^n_o$ and $S=\supp S_Q$, then
\begin{equation}\label{eq:reconstruct}
 Q=\bigcap_{u\in S}\{x:\langle x,u\rangle\le h_Q(u)\}.
\end{equation}
If $a\in C(\Sn)$ is positive and $Q=W[a]$, then
\begin{equation*}
 a=h_Q\quad\text{on }\supp S_Q.
\end{equation*}
\end{lemma}
\begin{proof}
The equilibrium identity $\int u\,dS_Q(u)=0$ and the fact that $S_Q$ is not concentrated on a proper great subsphere imply that $S$ positively spans $\R^n$. The right side of \eqref{eq:reconstruct} is therefore a convex body $P\supset Q$. Its support function agrees with $h_Q$ on $S$, so
\[
 V_1(Q,P)=\frac1n\int h_P\,dS_Q=|Q|.
\]
Minkowski's first inequality gives $|Q|\ge |Q|^{(n-1)/n}|P|^{1/n}$. Thus $|P|\le|Q|$, and the inclusion implies $P=Q$.

Fix $x\in\partial Q$.  At least one defining inequality is active at $x$.
Indeed, otherwise the positive continuous function
$u\mapsto a(u)-\langle x,u\rangle$ would have a positive minimum on the
sphere, placing a neighborhood of $x$ inside $Q$.  Thus there is a unit
vector $u$ with $\langle x,u\rangle=a(u)$.  Since $x\in Q$ and $h_Q\leq a$,
this gives $h_Q(u)=a(u)$, and $u$ is an outer normal at $x$.

At $\mathcal H^{n-1}$-almost every boundary point the outer unit normal is
unique.  At such a point the active direction must therefore be
$u=\nu_Q(x)$.  By the definition of surface area measure, $a=h_Q$ holds
$S_Q$-almost everywhere.  The nonnegative difference $a-h_Q$ is
continuous, so it vanishes on $\supp S_Q$.
\end{proof}

In particular, the reconstruction formula gives
\begin{equation}\label{eq:lp-reconstruction-consequence}
 h_L=c h_K\text{ on }\supp S_K
 \quad\Longrightarrow\quad L\subset cK\qquad(c>0).
\end{equation}

\begin{lemma}\label{lem:product}
Suppose $\R^n=E_1\oplus\cdots\oplus E_r$ is an orthogonal decomposition and $K_i\subset E_i$ contains the origin in its relative interior. Put $d_i=\dim E_i$, $v_i=|K_i|_{d_i}$, and $K=\bigoplus_i K_i$. Regard the intrinsic measure $V^{(d_i)}_{K_i}$ as a measure on $\Sn$ supported on $S(E_i)=E_i\cap\Sn$. Then
\begin{align}
 V_K&=\sum_{i=1}^r\frac{d_i}{n}
 \left(\prod_{j\ne i}v_j\right)V^{(d_i)}_{K_i},\notag
\\
 \vbar_K&=\sum_{i=1}^r\frac{d_i}{n}\frac{V^{(d_i)}_{K_i}}{v_i}.
 \label{eq:product-normalized}
\end{align}
For positive numbers $a_i$ and $L=\bigoplus_i a_iK_i$,
\begin{equation*}
 V_L=\left(\prod_i a_i^{d_i}\right)V_K,
 \qquad \vbar_L=\vbar_K.
\end{equation*}
These scaling identities hold for arbitrary complementary subspaces. In that case the normal subspaces are
\begin{equation}\label{eq:normal-subspaces}
 F_i=\left(\bigoplus_{j\ne i}E_j\right)^\perp,
 \qquad \supp V_K\subset\bigcup_i S(F_i).
\end{equation}
\end{lemma}
\begin{proof}
Up to a set of boundary area zero, the boundary of an orthogonal product is the union of the products with one boundary factor and all other factors interior. On the part corresponding to $K_i$, the normal lies in $E_i$, the supporting distance is $h_{K_i}$, and integration over the other factors contributes $\prod_{j\ne i}v_j$. The boundary representation of the cone-volume measure gives the first product formula. Division by $|K|=\prod_i v_i$ proves \eqref{eq:product-normalized}. Each term of the first formula acquires the common factor $\prod_i a_i^{d_i}$ under the stated dilations. This proves the scaling identities. For complementary subspaces choose a common linear map making them orthogonal and apply \eqref{eq:covariance}. The normal transformation also gives \eqref{eq:normal-subspaces}.
\end{proof}

For $C\in\mathcal K^n_o$, set
\begin{equation*}
 Z_C(\xi)=\int_C e^{-\langle\xi,x\rangle}\,dx,
 \qquad \xi_C=\argmin_{\xi\in\R^n}Z_C(\xi),
 \qquad \cP(C)=\log Z_C(\xi_C).
\end{equation*}

\begin{lemma}\label{lem:tilt}
The minimizer defining $\xi_C$ exists and is unique. It is characterized by
\begin{equation}\label{eq:tilt-centering}
 \int_C xe^{-\langle\xi_C,x\rangle}\,dx=0.
\end{equation}
Moreover,
\begin{equation}\label{eq:tilt-volume}
 \cP(C)\le\log|C|,
 \qquad \cP(C)=\log|C|\quad\Longleftrightarrow\quad g(C)=0.
\end{equation}
If $C_j\to C$ in the Hausdorff metric within $\mathcal K^n_o$, then $\xi_{C_j}\to\xi_C$ and $\cP(C_j)\to\cP(C)$. The functional $\cP$ is strictly increasing under proper inclusion and
\begin{equation*}
 \cP(aC)=\cP(C)+n\log a\qquad(a>0).
\end{equation*}
\end{lemma}
\begin{proof}
Differentiating the logarithmic partition function gives
\[
 \nabla_\xi\log Z_C(\xi)=-\E_\xi x,
 \qquad D^2_\xi\log Z_C(\xi)=\Cov_\xi(x)>0.
\]
Strict positivity follows from the nonempty interior of $C$. If $rB\subset C$ and $e=\xi/|\xi|$, then on $B(-re/2,r/4)$ one has $-\langle\xi,x\rangle\ge r|\xi|/4$. Consequently $Z_C(\xi)\ge |(r/4)B|e^{r|\xi|/4}$, which proves coercivity. Strict convexity gives uniqueness and \eqref{eq:tilt-centering}. Evaluation at $\xi=0$ proves the inequality in \eqref{eq:tilt-volume}; equality holds precisely when $\xi=0$ is the minimizer, equivalently $g(C)=0$.

For a convergent sequence of bodies, common inner and outer balls give a common coercive bound and an upper bound on $Z_{C_j}(0)$. Thus their minimizers remain in a compact set. The integrals converge uniformly on compact sets of $\xi$, by dominated convergence and the nullity of convex boundaries. Uniqueness identifies every limit of the minimizers. For a proper inclusion $K\subsetneq L$, use the minimizer for $L$:
\[
 \Psi(K)\le\int_K e^{-\xi_L\cdot x}\,dx
 <\int_L e^{-\xi_L\cdot x}\,dx=\Psi(L).
\]
The strict inequality holds because $L\setminus K$ has positive Lebesgue measure and the weight is positive. Taking logarithms proves strict monotonicity of $\cP$. A change of variables proves the scaling identity, equivalently $\Psi(aC)=a^n\Psi(C)$.
\end{proof}

For $a\in C(\Sn)$ define $\cH(a)=\cP(W[e^a])$. Lemma~\ref{lem:tilt} and \eqref{eq:wulff-sandwich} show, before any concavity assertion, that
\begin{equation}\label{eq:H-basic}
 \cH(a+c)=\cH(a)+nc,\qquad
 |\cH(a)-\cH(b)|\le n\|a-b\|_\infty.
\end{equation}
The functional is monotone in its argument. Since $Z_C(\xi_C)=\Psi(C)$, define the weighted cone probability by
\begin{equation}\label{eq:omega}
 \int_{\Sn} f\,d\omega_C
 =\frac1{nZ_C(\xi_C)}\int_{\partial C}
 f(\nu_C)h_C(\nu_C)e^{-\langle\xi_C,x\rangle}\,d\mathcal H^{n-1}(x).
\end{equation}
The divergence theorem applied to $xe^{-\langle\xi_C,x\rangle}$ and \eqref{eq:tilt-centering} gives $\omega_C(\Sn)=1$. Its support is $\supp S_C$, and $\omega_C=\vbar_C$ if $C$ is centered.

For the integral in \eqref{eq:lp-weighted-Vp}, this notation gives
\begin{equation*}
 \widetilde V_p(K,L)=\Psi(K)\int_{\Sn}
       \left(\frac{h_L}{h_K}\right)^p\,d\omega_K.
\end{equation*}
This identity is valid for every real $p$, since the support functions are positive and continuous. In particular, $\widetilde V_0(K,L)=\Psi(K)$.

\begin{lemma}\label{lem:H-variation}
For every $C\in\mathcal K^n_o$ and $f\in C(\Sn)$, the two-sided directional derivative exists and satisfies
\begin{equation}\label{eq:H-variation}
 D\cH(\log h_C)[f]=n\int_{\Sn}f\,d\omega_C.
\end{equation}
\end{lemma}
\begin{proof}
For fixed $\xi$, the weighted Aleksandrov formula gives
\[
 \left.\frac{d}{dt}\right|_{t=0}
 \int_{W[h_Ce^{tf}]}e^{-\langle\xi,x\rangle}\,dx
 =\int_{\partial C}f(\nu_C)h_C(\nu_C)e^{-\langle\xi,x\rangle}
 \,d\mathcal H^{n-1}(x).
\]
The multiplicative perturbation differs from $h_C+th_Cf$ by $O(t^2)$ uniformly. A dilation sandwich changes the corresponding weighted volumes by $O(t^2)$, so the additive first-variation formula applies.

Put $C_t=W[h_Ce^{tf}]$, $\xi_t=\xi_{C_t}$, and $J(t,\xi)=\log\int_{C_t}e^{-\langle\xi,x\rangle}\,dx$. Lemma~\ref{lem:tilt} gives $\xi_t\to\xi_C$. All bodies lie in a fixed ball and $|C_t\mathbin\triangle C|=O(|t|)$. Write $I(t,\xi)=\int_{C_t}e^{-\xi\cdot x}\,dx$. Uniform convergence of the weights gives
\[
 \big|[I(t,\xi_t)-I(0,\xi_t)]-[I(t,\xi_C)-I(0,\xi_C)]\big|
 \le |C_t\mathbin\triangle C|\,
       \sup_x|e^{-\xi_t\cdot x}-e^{-\xi_C\cdot x}|=o(|t|),
\]
where the supremum is over that fixed ball. The increments are $O(|t|)$ and $I(0,\xi_t)\to I(0,\xi_C)>0$. Expanding the logarithm therefore shows that the logarithmic difference quotients also have the same limit for the two parameters. Minimality gives
\begin{align*}
 J(t,\xi_t)-J(0,\xi_t)
 &\le\cH(\log h_C+tf)-\cH(\log h_C)\\
 &\le J(t,\xi_C)-J(0,\xi_C).
\end{align*}
After division by $t$, with the inequalities reversed if $t<0$, the two bounds have the common limit in \eqref{eq:H-variation}.
\end{proof}

\section{Proof of Theorem 1.1}\label{sec:proof-bm}\ 

We first record the analytic facts that are used for both functional concavity and its equality limit. The remaining argument is the proof of Theorem~\ref{thm:logbm}; its equality part establishes a stronger implication for affine segments of $\cH$, needed in Section~\ref{sec:proof-m-uniqueness}.

\subsection{Moment coordinates and analytic preliminaries}\label{subsec:moment}\ 

We begin with the whole-space moment map. The compact-target theorem of Berman--Berndtsson \cite[Theorem 1.1]{BB13}, together with the moment-measure theory of \cite{CEK15}, applies to a positive smooth density on a convex body whose weighted centroid is zero. It gives a smooth strictly convex moment potential with gradient image equal to the interior of the body, unique up to source translation. For a uniformly convex target potential on a ball, Klartag's estimate \cite[Proposition 3.1 and Remark 3.5]{Klartag13} bounds the moment Hessian by $4\varepsilon^{-1}I$. These results require the centroid condition, not evenness.

\begin{lemma}\label{lem:moment}
Let $V\in C^\infty(\R^n)$ satisfy $D^2V\succeq\varepsilon I$ for some $\varepsilon>0$, and suppose
\[
 Z=\int_{\R^n}e^{-V(x)}\,dx,\qquad
 d\mu(x)=Z^{-1}e^{-V(x)}\,dx,\qquad \int x\,d\mu(x)=0.
\]
There is a strictly convex $\varphi\in C^\infty(\R^n)$ such that
\begin{equation*}
 \int e^{-\varphi(z)}\,dz=1,\qquad
 (\nabla\varphi)_\#(e^{-\varphi(z)}\,dz)=\mu,\qquad
 0\prec D^2\varphi\preceq\frac4\varepsilon I.
\end{equation*}
The gradient $\nabla\varphi:\R^n\to\R^n$ is a diffeomorphism, $\varphi(z)\ge c|z|-b$ for some $c>0$, and
\begin{equation*}
 \log\det D^2\varphi(z)=V(\nabla\varphi(z))+\log Z-\varphi(z).
\end{equation*}
\end{lemma}
\begin{proof}
Uniform convexity gives Gaussian decay after an affine term, so $Z$ and all polynomial moments are finite. Set
\[
 Z_R=\int_{B_R}e^{-V(x)}\,dx,\qquad
 m_R=Z_R^{-1}\int_{B_R}xe^{-V(x)}\,dx.
\]
Then $m_R\to0$. Translate the truncated probability by $-m_R$:
\[
 d\mu_R(x)=Z_R^{-1}\ind_{B_R-m_R}(x)e^{-V(x+m_R)}\,dx.
\]
Its centroid is zero. The compact-target theorem and the Hessian estimate give a potential $\varphi_R$ with
\[
 (\nabla\varphi_R)_\#(e^{-\varphi_R}\,dz)=\mu_R,\quad
 \int e^{-\varphi_R}=1,\quad
 \nabla\varphi_R(\R^n)=B_R-m_R,\quad
 0\prec D^2\varphi_R\preceq\Lambda I,
 \qquad\Lambda=4/\varepsilon.
\]
Translate the source so that the minimum is at zero, and put $a_R=\varphi_R(0)$ and $g_R=\varphi_R-a_R$. Supporting planes with slopes surrounding zero show that each $\varphi_R$ grows at least linearly. Thus integration of $\diver(ze^{-\varphi_R})$, justified by expanding compact cutoffs, gives
\begin{equation*}
 \int z\cdot\nabla\varphi_R(z)e^{-\varphi_R(z)}\,dz=n.
\end{equation*}
For $x=\nabla\varphi_R(z)$, convexity and the minimum normalization imply $x\cdot z\ge g_R(z)\ge0$ and
\[
 g_R(y)\ge g_R(z)+x\cdot y-x\cdot z.
\]
Together with $g_R(y)\ge0$, this yields $g_R(y)\ge(x\cdot y)_+-x\cdot z$. Integrating and using the preceding radial identity, we obtain
\begin{equation}\label{eq:source-coercive}
 g_R(y)\ge\int(x\cdot y)_+\,d\mu_R(x)-n\ge c|y|-n.
\end{equation}
The constant $c>0$ is independent of large $R$: a fixed ball about zero lies in $B_R-m_R$, and the densities there have a common positive lower bound. The Hessian estimate gives $g_R(y)\le\Lambda|y|^2/2$ and $|\nabla\varphi_R(y)|\le\Lambda|y|$. Normalization now implies
\begin{equation*}
 (2\pi/\Lambda)^{n/2}\le e^{a_R}=\int e^{-g_R(y)}\,dy
 \le e^n\int e^{-c|y|}\,dy.
\end{equation*}
Hence $a_R$ is bounded and the source densities have a common integrable exponential majorant.

The equation for the compact targets is
\[
 \log\det D^2\varphi_R=V(\nabla\varphi_R+m_R)+\log Z_R-\varphi_R.
\]
On each fixed source ball its right side and first derivatives are uniformly bounded. The determinant is bounded below there, and the Hessian upper bound gives a positive lower bound for each eigenvalue. Interior Evans--Krylov and Schauder estimates therefore give bounds of every order on smaller balls. One may apply the uniformly elliptic form of these estimates by extending the logarithm from the fixed eigenvalue interval to a smooth concave uniformly elliptic spectral function. A diagonal subsequence converges locally smoothly to $\varphi$, with the asserted moment equation, the Hessian bounds, and the linear lower bound. Dominated convergence gives $\int e^{-\varphi}=1$. For every bounded continuous $F$,
\[
 \int F(\nabla\varphi)e^{-\varphi}
 =\lim_R\int F(\nabla\varphi_R)e^{-\varphi_R}
 =\lim_R\int F\,d\mu_R=\int F\,d\mu.
\]
This proves the pushforward assertion.

The positive Hessian makes the gradient locally invertible and injective. Its image is dense because $\mu$ has positive density everywhere. Fix $x_0\in\R^n$ and choose finitely many image points $y_j=\nabla\varphi(z_j)$ whose convex hull contains a ball $B_r(x_0)$. Their supporting planes give
\[
 \varphi(z)-x_0\cdot z
 \ge\max_j\{(y_j-x_0)\cdot z+\varphi(z_j)-y_j\cdot z_j\}
 \ge r|z|-b_{x_0}.
\]
The left side attains a minimum, where $\nabla\varphi=x_0$. Thus the gradient is onto, and its local inverses form a global smooth inverse.
\end{proof}

Write
\[
 d\nu(z)=e^{-\varphi(z)}\,dz,\quad x=\nabla\varphi(z),\quad
 \tau=D_z^2\varphi,\quad M=\tau^{-1}.
\]
In target coordinates $\tau(x)$ means $D_z^2\varphi((\nabla\varphi)^{-1}(x))$. Expectations in the two coordinates are identified by the moment map. Repeated coordinate indices are summed. Let
\[
 \delta_\mu W=\nabla V\cdot W-\diver_xW,\qquad
 \partial_k^*=x_k-\partial_{z_k}.
\]
On matrices, $\delta_\mu$ acts rowwise. Compact integration by parts gives the Stein identity
\begin{equation}\label{eq:stein}
 \E[x_jF]=\E[\tau_{jk}\partial_{x_k}F],\qquad \delta_\mu\tau=x,
\end{equation}
which is the moment-map construction in \cite{Fathi19}. Define
\begin{equation*}
 Av=\delta_\mu(\tau\nabla_xv)=x\cdot\nabla_xv-\tau:D_x^2v.
\end{equation*}
It is formally symmetric on compact tests and satisfies $A1=0$ and $Ax_j=x_j$.

For $v\in C_c^\infty(\R^n)$ put
\begin{equation}\label{eq:test-fields}
 h=D_x^2v,\quad P=\tau h,\quad B=\tau h\tau=P\tau,\quad
 W=\delta_\mu B,\quad u=\delta_\mu W.
\end{equation}
All these fields have compact support in both coordinate systems. The chain rule gives
\begin{equation}\label{eq:adjoints}
 W_a=\partial_k^*P_{ak},\qquad
 W=\nabla_z\big((A-I)v\circ\nabla\varphi\big),\qquad
 u=A(A-I)v.
\end{equation}
For example, with $C_{ijk}=\partial_{z_i}\partial_{z_j}\partial_{z_k}\varphi$, expansion of the first expression gives
\[
 W_a=x_k\tau_{ai}v_{x_ix_k}-C_{aik}v_{x_ix_k}
      -\tau_{ai}\tau_{jk}v_{x_ix_kx_j},
\]
which is also the derivative in $z_a$ of $(A-I)v$. The full symmetry of the third derivatives is the only symmetry used. Two adjunctions yield
\begin{equation}\label{eq:pairings}
 \E[uF]=\E[W\cdot\nabla_xF]=\E[B:D_x^2F].
\end{equation}
In particular, $\E u=0$ and $\E[xu]=0$. Finally,
\begin{equation}\label{eq:hessian-identity}
 \|Av\|_2^2-\E[vAv]
 =\E\big\|\tau^{1/2}D_x^2v\tau^{1/2}\big\|_\HS^2.
\end{equation}
Indeed, the left side is $\E\nabla_xv\cdot\tau\nabla_x(A-I)v=\E\nabla_xv\cdot W=\E h:B$. These compact identities are the nonsymmetric form of the adjoint calculation in \cite[Section 4]{OpenAI26}; no domain for an unbounded operator is asserted.

The next lemma identifies the full obstruction to density of the compact double-divergence tests.

\begin{lemma}\label{lem:density}
Suppose $d\mu=\rho\,dx$ has positive smooth density and finite second moment, $\tau$ is smooth with $0\prec\tau\preceq\Lambda I$, and $A=x\cdot\nabla-\tau:D^2$ satisfies compact weighted symmetry and \eqref{eq:hessian-identity}. Then
\begin{equation}\label{eq:dense-range}
 \overline{\{A(A-I)v:v\in C_c^\infty(\R^n)\}}^{\,L^2(\mu)}
 =\Span\{1,x_1,\ldots,x_n\}^{\perp}.
\end{equation}
\end{lemma}
\begin{proof}
Write $\Gamma(f,h)=\nabla f\cdot\tau\nabla h$ and $\Gamma(f)=\Gamma(f,f)$. Since $A1=0$ and $Ax_j=x_j$, compact symmetry makes every element of the displayed range orthogonal to affine functions. Conversely, suppose $g\in L^2(\mu)$ annihilates that range. With $\mathscr L=A(A-I)$, formal symmetry gives
\[
 \mathscr L^t(\rho g)=\rho\mathscr Lg=0
\]
in distributions. Its principal symbol is $(\zeta\cdot\tau(x)\zeta)^2$, which is locally uniformly positive. To apply the Lebesgue-self-adjoint regularity statement in Lemma~\ref{lem:finite-elliptic}, set
\[
 \widehat{\mathscr L}=\rho^{1/2}\mathscr L\rho^{-1/2},\qquad
 \widehat g=\rho^{1/2}g.
\]
Weighted symmetry makes $\widehat{\mathscr L}$ formally self-adjoint for Lebesgue measure. It has smooth coefficients and the same principal symbol, and $\widehat{\mathscr L}\widehat g=0$ distributionally. The lemma gives $\widehat g\in C^4_\loc$, hence $g\in C^4_\loc$ because $\rho$ is smooth and positive. Compact $C^4$ tests can be used by approximation.

Choose $0\le\eta_r\le1$, equal to one on $B_r$ and supported in $B_{2r}$, with $|\nabla\eta_r|\le C/r$ and $|D^2\eta_r|\le C/r^2$. For $r\ge1$, the Hessian cap and the linear drift give
\begin{equation*}
 \Gamma(\eta_r)^{1/2}\le B/r,\qquad |A\eta_r|\le K,
\end{equation*}
where $B,K$ do not depend on $r$. All derivatives vanish off $B_{2r}\setminus B_r$. Put $v=Ag$, initially as a local function, and set
\[
 G=\|g\|_2,\qquad X_r=\|\eta_r^2v\|_2,\qquad
 Y_r^2=\E[\eta_r^2\Gamma(g)].
\]
Testing $\mathscr Lg=0$ against $\eta_r^4g$ gives $\E[v(A-I)(\eta_r^4g)]=0$. The product rule $A(ab)=aAb+bAa-2\Gamma(a,b)$ yields
\begin{align*}
 (A-I)(\eta_r^4g)
 &=\eta_r^4v+g\{4\eta_r^3A\eta_r-12\eta_r^2\Gamma(\eta_r)-\eta_r^4\}
   -8\eta_r^3\Gamma(\eta_r,g).
\end{align*}
Cauchy--Schwarz and the cutoff bounds, together with integration by parts in $\E[\eta_r^2gAg]$, imply
\begin{equation*}
 X_r\le aG+\frac{8B}{r}Y_r,\qquad
 Y_r^2\le GX_r+\frac{2B}{r}GY_r,
 \qquad a=1+4K+12B^2.
\end{equation*}
The first estimate is immediate when $X_r=0$; otherwise divide by $X_r$. These estimates give $Y_r^2\le aG^2+10BGY_r$, hence uniform bounds on $X_r,Y_r$. Exhaustion by balls proves $v\in L^2(\mu)$. Now put $w=g-v$. The equation gives
\[
 Aw=0,\qquad Av=v,\qquad w,v\in L^2(\mu).
\]
Testing $Aw=0$ with $\eta_r^2w$ gives
\[
 \big(\E[\eta_r^2\Gamma(w)]\big)^{1/2}\le\frac{2B}{r}\|w\|_2.
\]
Let $r\to\infty$ on any fixed ball. Positivity of $\tau$ shows that $w$ is constant. Similarly, with $Z_r^2=\E[\eta_r^2\Gamma(v)]$, testing $Av=v$ gives
\[
 Z_r^2\le\|v\|_2^2+\frac{2B}{r}\|v\|_2Z_r.
\]
Thus $\E\Gamma(v)<\infty$. For $f_r=\eta_rv$, the product rule gives
\[
 Af_r=\eta_rv+vA\eta_r-2\Gamma(\eta_r,v),\qquad
 f_r\to v,\quad Af_r\to v\quad\hbox{in }L^2(\mu).
\]
The second term has norm at most $K\|v\ind_{B_{2r}\setminus B_r}\|_2\to0$, and the last has norm at most $2Br^{-1}(\E\Gamma(v))^{1/2}$. Apply \eqref{eq:hessian-identity}, valid for compact $C^2$ tests by approximation:
\[
 \E\big\|\tau^{1/2}D^2f_r\tau^{1/2}\big\|_\HS^2
 =\|Af_r\|_2^2-\E[f_rAf_r]\longrightarrow0.
\]
On each fixed ball $f_r=v$ for large $r$, so $D^2v=0$. Hence $v$ is affine, and $Av=v$ eliminates its constant term. Therefore $g=w+v$ is affine. Taking orthogonal complements proves \eqref{eq:dense-range}.
\end{proof}

The support constraints below involve $(s_+)^q$, which has finite rather than infinite differentiability. We give the required regularity argument instead of treating this function as smooth.

\begin{lemma}\label{lem:finite-elliptic}
Let $\Omega\subset\R^n$ be open and let $\mathscr L$ be a scalar fourth-order operator, formally self-adjoint for Lebesgue measure, with coefficients in $C^{2n+14}_\loc(\Omega)$. Suppose its principal symbol is positive and locally uniformly elliptic. If $w\in L^2_\loc(\Omega)$ and $\mathscr Lw=0$ distributionally, then
\[
 w\in H^{n+8}_\loc(\Omega)\subset C^4_\loc(\Omega).
\]
\end{lemma}
\begin{proof}
Localize in a coordinate patch and place it inside a flat torus. With a smooth cutoff $\theta$ supported in the patch, extend the operator by
\[
 \widetilde{\mathscr L}=\theta\mathscr L\theta+(1-\theta)\Delta^2(1-\theta).
\]
The first term is extended by zero outside the patch. Both terms are formally self-adjoint. The principal symbol is $\theta^2p_4+(1-\theta)^2|\zeta|^4$, hence uniformly positive. The coefficients remain $C^N$, where $N=2n+14$, and the extension agrees with $\mathscr L$ where $\theta=1$.

On the torus of period $2\pi$, write $p_4$ for the extended principal symbol and put
\[
 b(x,k)=p_4(x,k)^{-1}\ (k\ne0),\quad b(x,0)=0,\qquad
 Qv(x)=\sum_{k\in\mathbb Z^n}b(x,k)\widehat v(k)e^{ik\cdot x}.
\]
Uniform ellipticity gives $|\partial_x^\alpha b(x,k)|\le C_\alpha\langle k\rangle^{-4}$ for $|\alpha|\le N$, where $\langle k\rangle=(1+|k|^2)^{1/2}$. The finite product rule gives
\[
 \widetilde{\mathscr L}Q=I+R,
 \qquad Rv=\sum_k r(x,k)\widehat v(k)e^{ik\cdot x},
\]
with $|\partial_x^\alpha r(x,k)|\le C_\alpha\langle k\rangle^{-1}$ for $|\alpha|\le N-4$; at $k=0$ set $r(x,0)=-1$. Let $d=N-4=2n+10$. Taking Fourier coefficients in $x$ and integrating by parts with $(1-\Delta_x)^{d/2}$ yields
\[
 |\widehat b(m,k)|\le C\langle m\rangle^{-d}\langle k\rangle^{-4},\qquad
 |\widehat r(m,k)|\le C\langle m\rangle^{-d}\langle k\rangle^{-1}.
\]
The Fourier matrix of $R^*$ therefore satisfies
\[
 |(R^*)_{\ell k}|\le C\langle\ell-k\rangle^{-d}\langle\ell\rangle^{-1}.
\]
Using $\langle\ell\rangle^s/\langle k\rangle^s\le C_s\langle\ell-k\rangle^{|s|}$ and Young's convolution inequality, we obtain
\begin{equation*}
 R^*:H^s\to H^{s+1}\quad(d>n+|s|),\qquad
 Q^*:H^t\to H^{t+4}\quad(d>n+|t|).
\end{equation*}
This also holds for negative exponents. Formal self-adjointness gives
\begin{equation*}
 Q^*\widetilde{\mathscr L}=I+R^*.
\end{equation*}
The identity extends by density to $H^s$ for integer $0\le s\le n+7$. Indeed, the same Fourier-matrix estimate makes $\widetilde{\mathscr L}:H^s\to H^{s-4}$ continuous; the coefficient condition $N>n+|s-4|$ holds throughout this range.

Suppose inductively that $w\in H^s_\loc$. Choose nested compact cutoffs $\chi,\chi_1$ where $\theta=1$, with $\chi_1=1$ near $\supp\chi$. Then
\[
 \widetilde{\mathscr L}(\chi w)=[\mathscr L,\chi](\chi_1w)\in H^{s-3}.
\]
The commutator has order at most three. Applying the adjoint parametrix identity and the preceding Sobolev mapping bounds gives
\[
 \chi w=Q^*[\mathscr L,\chi](\chi_1w)-R^*(\chi w)\in H^{s+1}.
\]
Both mapping conditions hold for $0\le s\le n+7$. Starting from $s=0$ and using successively nested cutoffs proves the claimed Sobolev regularity. Since $n+8>n/2+4$, Sobolev embedding gives $C^4_\loc$.
\end{proof}

The equality limit can have a nonzero exponential tilt. The following local convergence statement includes this case and allows the global convexity constants to tend to zero.

\begin{lemma}\label{lem:stability}
Let $C\in\mathcal K^n_o$, $\xi\in\R^n$, and
\[
 Z=\int_C e^{-\xi\cdot x}\,dx,\qquad
 d\mu=Z^{-1}\ind_C e^{-\xi\cdot x}\,dx,\qquad \int x\,d\mu=0.
\]
Let $V_j$ be convex potentials with $D^2V_j\succeq\varepsilon_jI$ for some $\varepsilon_j>0$, and let
\[
 Z_j=\int e^{-V_j},\qquad d\mu_j=Z_j^{-1}e^{-V_j}\,dx.
\]
Assume that every $\mu_j$ is centered and
\begin{equation*}
 W_1(\mu_j,\mu)\to0,\qquad Z_j\to Z,\qquad
 V_j\to(x\mapsto\xi\cdot x)\quad\hbox{locally in every fixed }C^k(\intt C).
\end{equation*}
For finite-regularity potentials, their differentiability orders are understood to tend to infinity. Assume that their probability-normalized moment potentials $\varphi_j$ have positive definite Hessians and gradient diffeomorphisms on $\R^n$. After source translations, for every fixed finite $k$,
\begin{equation*}
 \varphi_j\to\varphi\quad\hbox{in }C^k_\loc(\R^n),
\end{equation*}
where $\varphi$ is the smooth strictly convex moment potential of $\mu$ and
\begin{equation*}
 \det D^2\varphi=Z\exp(-\varphi+\xi\cdot\nabla\varphi),
 \qquad\nabla\varphi:\R^n\longrightarrow\intt C
 \quad\hbox{is a diffeomorphism}.
\end{equation*}
Consequently,
\begin{equation*}
 D^2\varphi_j\circ(\nabla\varphi_j)^{-1}
 \longrightarrow D^2\varphi\circ(\nabla\varphi)^{-1}
 \quad\hbox{locally in every fixed }C^k(\intt C).
\end{equation*}
\end{lemma}
\begin{proof}
The compact-target theorem \cite[Theorem 1.1]{BB13}, applied to the positive smooth density $Z^{-1}e^{-\xi\cdot x}$ on $C$, gives the stated moment equation and uniqueness up to source translation. The required weighted centroid condition is exactly the hypothesis. Integrating through the gradient map verifies $\int e^{-\varphi}=1$. In particular, $e^{-\varphi}$ is a continuous strictly positive log-concave density on all of $\R^n$, without a regularity assumption on $\partial C$.

The measures $\mu_j,\mu$ are centered probabilities, have finite first moments, and are not supported on hyperplanes. Their moment potentials are essentially continuous. The stability theorem of Bonnet--Rubinstein \cite[Theorem 1.3]{BR26}, with $\mu$ fixed, gives translations $v_j$ such that
\begin{equation*}
 \|e^{-\varphi_j(\,\cdot-v_j)}-e^{-\varphi}\|_{L^1(\R^n)}\longrightarrow0.
\end{equation*}
If a target equals $\mu$, the same conclusion follows from translation uniqueness. Replace the potentials by these translates. This does not change the target-coordinate moment Hessian.

Convergence of log-concave densities to a continuous positive density is uniform on compact sets \cite[Proposition 2]{CS10}. Taking logarithms after this locally uniform convergence therefore gives local uniform convergence of $\varphi_j$. Convexity upgrades this to local uniform convergence of the gradients. Explicitly, for a compact set $Q$ with a compact neighborhood $Q^+$ and sufficiently small $s>0$, directional difference quotients give
\[
 \sup_Q|\nabla\varphi_j-\nabla\varphi|
 \le C_n\left(\frac{2\|\varphi_j-\varphi\|_{C^0(Q^+)}}s
                 +\omega_{Q^+}(s)\right),
\]
where $\omega_{Q^+}$ is a modulus of continuity of $\nabla\varphi$. Letting first $j\to\infty$ and then $s\downarrow0$ proves
\begin{equation*}
 \varphi_j\to\varphi\quad\hbox{in }C^1_\loc(\R^n).
\end{equation*}

We obtain uniform interior estimates without using the possibly diverging cap $4/\varepsilon_j$. The moment equation is
\begin{equation}\label{eq:approx-moment-equation}
 \det D^2\varphi_j=Z_j\exp\big(-\varphi_j+V_j(\nabla\varphi_j)\big).
\end{equation}
Fix a source ball $B_{2R}(z_0)$. Its closure has gradient image under $\varphi$ compactly contained in $\intt C$. By the local $C^1$ convergence, all its images under $\nabla\varphi_j$, for large $j$, lie in a common compact set $Q'\Subset\intt C$. Thus the hypotheses and \eqref{eq:approx-moment-equation} give
\begin{equation*}
 0<a\le\det D^2\varphi_j\le b<\infty
 \quad\hbox{on }B_{2R}(z_0),
\end{equation*}
with uniform bounds on every fixed derivative of $V_j$ on $Q'$.

The strict-convexity moduli can also be chosen uniformly. On the closure of this source ball let $D^2\varphi\succeq mI$, $m>0$, and put $e_j=\|\nabla\varphi_j-\nabla\varphi\|_{C^0}$. For $z,w$ in the ball,
\begin{align*}
 B_j(z,w)&:=\varphi_j(w)-\varphi_j(z)-\nabla\varphi_j(z)\cdot(w-z)\\
 &\ge\frac m2|w-z|^2-2e_j|w-z|.
\end{align*}
At every fixed separation $|w-z|\ge\rho>0$, this is at least $m\rho^2/4$ for all sufficiently large $j$. The finitely many remaining indices have a positive minimum on the same compact set of pairs. Hence the strict-convexity modulus is uniform at each fixed scale.

To specify the section geometry, subtract the tangent plane at $z_0$ and write
\[
 u_j(z)=\varphi_j(z)-\varphi_j(z_0)-\nabla\varphi_j(z_0)\cdot(z-z_0),
 \qquad S_j(t)=\{u_j<t\}.
\]
Using the strictly convex limit, choose $h,r>0$ such that its corresponding function $u$ satisfies $\min_{\partial B_R(z_0)}u>8h$ and $\sup_{B_r(z_0)}u<h/8$. For large $j$, local uniform convergence and convexity give
\[
 B_r(z_0)\subset S_j(h/4)\subset S_j(h/2)\subset S_j(h)
 \subset S_j(2h)\Subset B_R(z_0).
\]
The gradients of $u_j$ are bounded by a common $L>0$, and therefore
\[
 \operatorname{dist}(S_j(s),\partial S_j(t))\ge(t-s)/L
 \qquad(0<s<t\le2h).
\]
Caffarelli's interior section estimates, in the framework recalled in \cite[Sections 2.9 and 2.12]{BB13}, give a uniform local $C^{1,\alpha}$ bound from the determinant bounds and this section geometry. The right side of \eqref{eq:approx-moment-equation} is then uniformly H\"older continuous. The interior estimate of Figalli--Jhaveri--Mooney \cite[Theorem 1.1]{FJM16} gives a uniform local $C^{2,\alpha}$ bound on smaller neighborhoods, with the exponent decreased if necessary. Its constants are controlled by the determinant bounds, the H\"older bound, and the strict-convexity modulus just obtained. A finite covering gives, on every source compact set $Q$,
\begin{equation*}
 \delta_Q I\preceq D^2\varphi_j\preceq M_Q I
\end{equation*}
for positive constants independent of large $j$.

Let $A_j=(D^2\varphi_j)^{-1}$ and $w_{j,k}=\partial_k\varphi_j$. Differentiating the logarithm of \eqref{eq:approx-moment-equation} gives
\[
 (A_j)_{ab}\partial_{ab}w_{j,k}
 =\nabla V_j(\nabla\varphi_j)\cdot\nabla w_{j,k}-w_{j,k}.
\]
This equation is uniformly elliptic on compact sets. Schauder estimates and iteration yield bounds of every fixed order, using the corresponding derivatives of $V_j$ on $Q'$. For finite-regularity potentials, only the tail with sufficiently many derivatives is used. Compactness and the local $C^1$ convergence identify every limit and prove the asserted local $C^k$ convergence.

The inverse image of a target compact set under $\nabla\varphi$ is compact. The positive limiting Hessian and the local $C^k$ convergence give convergent local inverse maps there; injectivity identifies them with the global inverses. A finite covering and composition with the Hessians give the asserted convergence in target coordinates.
\end{proof}

\subsection{Concavity of the optimized functional}\label{subsec:concavity}\ 

\begin{proof}[Proof of the inequality in Theorem~\ref{thm:logbm}]
We prove the stronger assertion
\begin{equation}\label{eq:H-concavity}
 \cH((1-t)a+tb)\ge(1-t)\cH(a)+t\cH(b)
 \qquad(a,b\in C(\Sn),\ 0\le t\le1).
\end{equation}
We first justify the moment-coordinate calculations for the exact one-sided powers used to approximate support constraints. Fix $q\ge2n+20$ be even, $\varepsilon>0$, $c_i>0$, and $p_i\ne0$. Suppose the probability with potential
\begin{equation*}
 V(x)=\frac\varepsilon2|x|^2+\sum_{i=1}^N c_i\langle p_i,x\rangle_+^q+\xi\cdot x
\end{equation*}
is centered. Lemma~\ref{lem:moment} holds with $\varphi\in C^{q+1,\alpha}_\loc$ and a $C^{q,\alpha}_\loc$ gradient diffeomorphism, for every $0<\alpha<1$. The target field $\tau$ belongs to $C^{q-1,\alpha}_\loc$ and satisfies the same upper bound $4\varepsilon^{-1}I$. The compact identities \eqref{eq:stein}--\eqref{eq:hessian-identity}, the dense-range assertion \eqref{eq:dense-range}, and the tensor identities used below hold for this exact potential.

To establish these assertions, convolve each one-dimensional function $(s_+)^q$ with a nonnegative compactly supported smooth probability kernel of width $\delta$. The resulting convex nonnegative summands converge in $C^{q-1,\alpha}$ on compact sets, since $(s_+)^q\in C^{q-1,1}_\loc$. Denote the smoothed untilted potential by $U_\delta$ and choose $\xi_\delta$ minimizing $\int e^{-U_\delta-\xi\cdot x}\,dx$. Existence, uniqueness, and centering follow as in Lemma~\ref{lem:tilt}: the partition function is strictly convex, the potentials are uniformly bounded on a fixed ball, and integration on a ball opposite the tilt direction gives a common coercive bound. Also $\int e^{-U_\delta}\le\int e^{-\varepsilon|x|^2/2}$. The minimizers are therefore bounded. Gaussian domination and uniqueness imply $\xi_\delta\to\xi$, local $C^{q-1,\alpha}$ convergence of the target densities, and convergence of all polynomial moments.

Apply Lemma~\ref{lem:moment} to the smoothed centered probabilities and translate the source variables so that the minima occur at the origin. The proof of \eqref{eq:source-coercive} gives common source coercivity, because the target densities have a common positive lower bound on a fixed ball. The global Hessian cap and the moment equation give local positive lower Hessian bounds. Interior regularity and Schauder bootstrap then give uniform local $C^{q+1,\alpha}$ estimates, first with a slightly larger H\"older exponent and then with that exponent decreased. Passing to a subsequence proves the equation, normalization, pushforward, and the asserted regularity. The dense-image and supporting-plane argument from Lemma~\ref{lem:moment} proves surjectivity of the limiting gradient. Local inverse regularity yields the assertion about $\tau$.

On the support of a fixed smooth compact test, all coefficients and derivatives occurring in the compact adjoint and tensor calculations converge. The identities therefore pass to the exact potential. In particular, homogeneity is exact after this limit.

For the dense-range assertion, let $\rho=d\mu/dx$ and $\mathscr L=A(A-I)$. Its coefficients use at most two derivatives of $\tau$ and belong to $C^{q-3,\alpha}_\loc$. The conjugate
\[
 \widetilde{\mathscr L}=\rho^{1/2}\mathscr L\rho^{-1/2}
\]
is formally self-adjoint for Lebesgue measure. Its coefficients use at most four derivatives of $\rho^{-1/2}$ and belong to $C^{q-5,\alpha}_\loc\subset C^{2n+14}_\loc$. Its principal symbol is $(\zeta^T\tau\zeta)^2$. If $g\in L^2(\mu)$ annihilates the range, then $w=\rho^{1/2}g$ solves $\widetilde{\mathscr L}w=0$ distributionally. This follows first for compact $C^4$ tests by approximation, for which all coefficient products are classical. Lemma~\ref{lem:finite-elliptic} gives $g\in C^4_\loc$ and $Ag\in C^2_\loc$. These are precisely the regularities used in the cutoff proof of Lemma~\ref{lem:density}. Compact $C^4$ equation tests and compact $C^2$ Hessian tests are permitted by smooth approximation. Thus the same argument proves \eqref{eq:dense-range}.

Let $V$ be a potential of the form just regularized, and write
\[
 V=\frac\varepsilon2|x|^2+\sum_i\psi_i+\xi\cdot x,\qquad
 d\mu=Z^{-1}e^{-V}\,dx,\qquad \E x=0.
\]
Each $\psi_i$ is nonnegative, convex, and positively homogeneous of degree $q$. All polynomial moments are finite. The preceding approximation supplies the moment coordinates and the dense test range. In this subsection expectations are with respect to $\mu$, or equivalently to its source probability. Put
\[
 H=D_x^2V=\varepsilon I+\sum_iH_i,\qquad H_i=D_x^2\psi_i,\qquad
 \Sigma=\E[xx^T]\succ0.
\]
For real $b_i$, define
\begin{equation*}
 \begin{gathered}
 F=\sum_i b_i\psi_i,\qquad \beta=\Sigma^{-1}\E[xF],\qquad
 U=F-\E F-\beta\cdot x,\\
 a_i=\E\psi_i,\qquad \gamma_i=b_i/q,\qquad
 C_b=\sum_i b_i^2a_i,\qquad D=C_b-\E U^2.
 \end{gathered}
\end{equation*}
Thus $U$ is the orthogonal projection of $F$ onto the complement of affine functions.

For the compact fields in \eqref{eq:test-fields}, set
\begin{equation*}
 \mathcal T(v)=\E\tr((D_xW)^2)-\E\tr(MBHB).
\end{equation*}
The first square is a matrix square, not the squared Hilbert--Schmidt norm. Define the invariant derivative norm
\begin{equation}\label{eq:D-norm}
 \cD(P)=\sum_{a,b,k,r,\ell,s}
 M_{ab}\tau_{kr}M_{\ell s}
 (\partial_{z_\ell}P_{ak})(\partial_{z_s}P_{br}).
\end{equation}

We claim that the exact remainder is the following sum of nonnegative terms. 
For every $v\in C_c^\infty(\R^n)$,
\begin{align}
 D+\|u-U\|_{L^2(\mu)}^2
 &=\mathcal T(v)+\varepsilon\E\{\tr(MB^2)+|W|^2\}\notag\\
 &\quad+\sum_i\E\tr\big(M(B-\gamma_i\tau)H_i(B-\gamma_i\tau)\big)\notag\\
 &\quad+\sum_i\E\big[(W-\gamma_ix)^TH_i(W-\gamma_ix)\big].
 \label{eq:exact-remainder}
\end{align}
Every term on the right is nonnegative. More precisely,
\begin{equation}\label{eq:derivative-remainder}
 \mathcal T(v)\ge\frac12\E\cD(P).
\end{equation}

To prove the claim, first weighted integration by parts gives
\begin{equation}\label{eq:vector-bochner}
 \E u^2=\E[W^THW]+\E\tr((D_xW)^2).
\end{equation}
Indeed, $\partial_{x_i}u=H_{ij}W_j+(V_{x_j}-\partial_{x_j})\partial_{x_i}W_j$, and pairing with $W_i$ proves the identity.

We next compute $\mathcal T(v)$. The calculation is the moment-coordinate tensor argument of \cite[Section 6]{OpenAI26}, with its full derivative remainder retained. Put $C_{ijk}=\varphi_{z_iz_jz_k}$, $(C_i)_{k\ell}=C_{ik\ell}$, and $\zeta_i=\tr(MC_i)$. Differentiating the moment equation gives
\[
 \tau_{ik}V_{x_k}=x_i+\zeta_i,
\]
and
\[
 (\tau H\tau)_{ij}
 =\tau_{ij}+M_{k\ell}\varphi_{z_kz_\ell z_iz_j}
 -\tr(MC_iMC_j)-C_{ijk}V_{x_k}.
\]
Using $\partial_{z_j}M=-MC_jM$ and the symmetry of $C$, we have
\[
 \partial_{z_d}M_{dk}=-\zeta_bM_{bk},\qquad
 \partial_{z_d}M_{dk}-x_dM_{dk}=-V_{x_k}.
\]
Consequently,
\begin{equation*}
 (\tau H\tau-\tau)_{ij}
 =(\partial_{z_d}-x_d)(M_{dk}C_{ijk})-\tr(MC_iMC_j).
\end{equation*}
The commutator $[\partial_{z_\ell},\partial_k^*]=\tau_{k\ell}$ and \eqref{eq:adjoints} give
\begin{equation*}
 D_zW=B+R,\qquad R_{a\ell}=\partial_k^*\partial_{z_\ell}P_{ak}.
\end{equation*}
Both $D_zW$ and $B$ are symmetric, so $R$ is symmetric. Since $D_xW=(B+R)M$ and $MBM=h$, expansion followed by one integration by parts yields
\begin{align*}
 \E\tr((D_xW)^2)
 &=\E\tr(MBMB)
 +2\E[(\partial_{z_k}h_{a\ell})(\partial_{z_\ell}P_{ak})]
 +\E\tr(MRMR).
\end{align*}
For $Q=MBMBM=h\tau h$, contraction of the preceding divergence identity and integration of its divergence term give
\begin{align*}
 \E\tr(MBHB)
 &=\E\tr(MBMB)-\E[(\partial_{z_d}Q_{ij})M_{dk}C_{ijk}]
 -\E[Q_{ij}\tr(MC_iMC_j)].
\end{align*}
All fields being differentiated in these integrations have compact support.

The contracted scalars in these two expansions are invariant under a constant dual change of coordinates
\[
 z=Tz',\qquad x'=T^Tx,\qquad v'(x')=v(T^{-T}x').
\]
The normalized potentials are $\varphi'(z')=\varphi(Tz')-\log|\det T|$ and $V'(x')=V(T^{-T}x')$, with $Z'=|\det T|Z$. In particular,
\[
 \tau'=T^T\tau T,\quad M'=T^{-1}MT^{-T},\quad
 P'=T^TPT^{-T},\quad B'=T^TBT,\quad R'=T^TRT.
\]
At a fixed point choose the constant matrix $T=\tau^{-1/2}$. Thus $\tau=I$ in the new coordinates at that point. This normalization is performed after the integrations by parts; no varying frame is differentiated.

In these normalized coordinates set
\[
 S_{ijk}=v_{x_ix_jx_k},\qquad J_{ijk}=h_{ia}C_{ajk},\qquad
 J^{(12)}_{ijk}=J_{jik}.
\]
The tensor $S$ is fully symmetric and $J$ is symmetric in its last two indices. Direct differentiation of $P=\tau h$ and $Q=h\tau h$ gives
\begin{align*}
 2(\partial_{z_k}h_{a\ell})(\partial_{z_\ell}P_{ak})
 &=2|S|^2+2\langle S,J\rangle,\\
 (\partial_{z_d}Q_{ij})C_{ijd}
 &=2\langle S,J\rangle+\langle J,J^{(12)}\rangle,
 \\
 Q_{ij}\tr(C_iC_j)&=|J|^2.
\end{align*}
For the first contraction, use $\partial_{z_k}h_{a\ell}=S_{a\ell k}$ and $\partial_{z_\ell}P_{ak}=C_{ai\ell}h_{ik}+S_{ak\ell}$. In the second contraction, differentiating the two factors $h$ gives the two copies of $\langle S,J\rangle$, while differentiating $\tau$ gives
\[
 \sum_{i,a,b,j,d}h_{ia}C_{abd}h_{bj}C_{ijd}
 =\sum_d\tr(hC_dhC_d)=\langle J,J^{(12)}\rangle.
\]
The last contraction follows from $Q=h^2$ and the symmetry of $C$. Subtracting the two expansions therefore yields
\begin{align*}
 \mathcal T(v)=\E\big[\|M^{1/2}RM^{1/2}\|_\HS^2
 +2|S|^2+4\langle S,J\rangle+|J|^2+\langle J,J^{(12)}\rangle\big].
\end{align*}
The last tensor expressions are evaluated in a normalized frame; their values do not depend on the choice of that frame. Full symmetrization gives
\[
 (\Sym J)_{ijk}=\frac{J_{ijk}+J_{jik}+J_{kij}}3,
 \qquad |\Sym J|^2=\frac{|J|^2+2\langle J,J^{(12)}\rangle}3,
 \qquad\langle S,\Sym J\rangle=\langle S,J\rangle.
\]
Hence
\begin{equation*}
 \mathcal T(v)=\E\left[
 \|M^{1/2}RM^{1/2}\|_\HS^2
 +2|S+\Sym J|^2+\frac16|J-J^{(12)}|^2\right].
\end{equation*}
For $K_{ak\ell}=\partial_{z_\ell}P_{ak}$ in normalized coordinates,
\[
 K_{ak\ell}=S_{ak\ell}+J_{ka\ell},\qquad
 |K|^2=|S|^2+2\langle S,J\rangle+|J|^2,\qquad
 \Sym K=S+\Sym J.
\]
The preceding formulas imply
\[
 2|S+\Sym J|^2+\frac16|J-J^{(12)}|^2
 =\frac12|K|^2+\frac32|\Sym K|^2.
\]
The invariant expression for $|K|^2$ is \eqref{eq:D-norm}. This proves \eqref{eq:derivative-remainder} and $\mathcal T(v)\ge0$.

It remains to complete the homogeneous squares. Let
\[
 m_i=\E\tr(\tau H_i),\qquad d_i=\E[u\psi_i].
\]
The Stein identity, homogeneity, and \eqref{eq:pairings} give
\begin{equation*}
 \begin{gathered}
 m_i=\E[x\cdot\nabla\psi_i]=qa_i,\qquad
 d_i=\E\tr(BH_i)=\E[W\cdot\nabla\psi_i],\\
 \E[x^TH_ix]=(q-1)m_i,\qquad
 \E[x^TH_iW]=(q-1)d_i.
 \end{gathered}
\end{equation*}
The first identity, involving a noncompact polynomial-growth test, is justified by an expanding cutoff, the upper bound on $\tau$, and the finite polynomial moments. The other integrations involve compact fields. Expanding gives
\begin{align*}
 \E\tr\big(M(B-\gamma_i\tau)H_i(B-\gamma_i\tau)\big)
 &=\E\tr(MBH_iB)-2\gamma_id_i+\gamma_i^2m_i,\\
 \E[(W-\gamma_ix)^TH_i(W-\gamma_ix)]
 &=\E[W^TH_iW]-2\gamma_i(q-1)d_i+\gamma_i^2(q-1)m_i.
\end{align*}
Their sum is $\E\tr(MBH_iB)+\E[W^TH_iW]-2b_id_i+b_i^2a_i$. Summing and using \eqref{eq:vector-bochner} shows that the right side of \eqref{eq:exact-remainder} is $C_b-2\E[uF]+\E u^2$. Since $u$ is orthogonal to affine functions, $\E[uF]=\E[uU]$, and this is $D+\|u-U\|_2^2$.

Finally, $H_i\succeq0$, $M\succ0$, and $B-\gamma_i\tau$ is symmetric, so
\[
 \tr\big(M(B-\gamma_i\tau)H_i(B-\gamma_i\tau)\big)
 =\|H_i^{1/2}(B-\gamma_i\tau)M^{1/2}\|_\HS^2\ge0.
\]
Also $\tr(MB^2)=\|BM^{1/2}\|_\HS^2\ge0$. This proves every sign assertion. The affine term $\xi\cdot x$ has zero Hessian and introduces no additional term.

The identity implies the projected variance bound. For the data above,
\begin{equation}\label{eq:variance}
 \Var_\mu(F)-\E[xF]^T\Sigma^{-1}\E[xF]
 \le\sum_i b_i^2\E\psi_i.
\end{equation}
For every $\eta>0$, a compact smooth test $v$ can be chosen so that its fields satisfy
\begin{equation}\label{eq:retained-fields}
 \|u-U\|_2\le\eta,\qquad
 \mathcal T(v)+\varepsilon\E|W|^2\le D+\eta^2,
 \qquad \frac12\E\cD(P)\le\mathcal T(v).
\end{equation}

Indeed, by construction $\E U=0$ and $\E[xU]=0$. Lemma~\ref{lem:density}, extended to the exact powers above, provides tests with $u\to U$ in $L^2(\mu)$. The nonnegative remainder in \eqref{eq:exact-remainder} gives $D\ge0$ in the limit. Since $\E U^2=\Var(F)-\E[xF]^T\Sigma^{-1}\E[xF]$, this is \eqref{eq:variance}. Keeping a test with error at most $\eta$ gives \eqref{eq:retained-fields}.

For later use, if $\kappa I\preceq\tau\preceq\Lambda I$ on a set, then
\begin{equation}\label{eq:derivative-comparison}
 \frac\kappa{\Lambda^2}|D_zP|^2\le\cD(P)
 \le\frac\Lambda{\kappa^2}|D_zP|^2.
\end{equation}
Only local upper and lower Hessian bounds are needed for this comparison.

We now pass from the variance bound to the optimized partition function. Fix unit vectors $u_1,\ldots,u_N$ that positively span $\R^n$. For $a=(a_i)\in\R^N$ put
\[
 Q(a)=\{x:\langle u_i,x\rangle\le e^{a_i},\ 1\le i\le N\},
 \qquad\psi_{i,q,a}(x)=(e^{-a_i}\langle u_i,x\rangle_+)^q.
\]
For $\varepsilon>0$ and even $q\ge2n+20$ define
\begin{equation}\label{eq:soft-partition}
 \begin{split}
 Z_{q,\varepsilon}(a,\xi)
 &=\int_{\R^n}\exp\left(-\frac\varepsilon2|x|^2
                    -\sum_i\psi_{i,q,a}(x)-\xi\cdot x\right)\,dx,\\
 J_{q,\varepsilon}(a,\xi)&=\log Z_{q,\varepsilon}(a,\xi),\qquad
 H_{q,\varepsilon}(a)=\inf_\xi J_{q,\varepsilon}(a,\xi).
 \end{split}
\end{equation}
On a compact set of $a$, there are common $r,R>0$ such that
\begin{equation}\label{eq:polytope-balls}
 rB\subset Q(a)\subset RB,\qquad
 \max_i e^{-a_i}\langle u_i,x\rangle_+\ge |x|/R.
\end{equation}
The integrand and all fixed parameter derivatives are integrable locally uniformly in $(a,\xi)$. The logarithmic partition function is strictly convex in $\xi$, with gradient $-\E x$ and Hessian $\Cov(x)\succ0$. For $e=\xi/|\xi|$, integration on $B(-re/4,r/8)$ gives
\begin{equation}\label{eq:soft-coercivity}
 Z_{q,\varepsilon}(a,\xi)
 \ge |(r/8)B|\exp\left(\frac r8|\xi|-N2^{-q}-\frac{\varepsilon r^2}{8}\right).
\end{equation}
Thus the infimum is achieved by a unique $\xi_{q,\varepsilon}(a)$, which centers the normalized probability. The implicit function theorem makes the minimizing tilt and $H_{q,\varepsilon}$ smooth in $a$.

Let $a(s)=a+sf$ and take expectations with respect to its optimally tilted probability. Set
\[
 F_s=-q\sum_i f_i\psi_{i,q,a(s)},\qquad
 \Sigma_s=\E_s[xx^T],\qquad \beta_s=\Sigma_s^{-1}\E_s[xF_s].
\]
Differentiating the centering equation gives
\[
 0=-\Cov_s\big(x,F_s+\xi'(s)\cdot x\big),\qquad \xi'(s)=-\beta_s.
\]
The envelope derivative and its second derivative are therefore
\begin{align*}
 \frac{d}{ds}H_{q,\varepsilon}(a(s))&=-\E_sF_s,
 \\
 \frac{d^2}{ds^2}H_{q,\varepsilon}(a(s))
 &=\Var_s(F_s)-\E_s[xF_s]^T\Sigma_s^{-1}\E_s[xF_s]
   -q^2\sum_i f_i^2\E_s\psi_{i,q,a(s)}.
\end{align*}
The estimate \eqref{eq:variance}, with $b_i=-qf_i$, makes the displayed second derivative nonpositive. Hence $H_{q,\varepsilon}$ is concave. The covariance subtraction is exactly the Schur complement produced by optimizing the tilt.

Radial integration by parts gives
\begin{equation*}
 q\sum_i\E_s\psi_{i,q,a(s)}=n-\varepsilon\E_s|x|^2.
\end{equation*}
The tilt term vanishes because $\E_sx=0$. The tails in \eqref{eq:polytope-balls} justify the integration and imply $|\frac d{ds}H_{q,\varepsilon}(a(s))|\le n\|f\|_\infty$.

We next remove the soft constraints. For fixed normals and $a$, as $q\to\infty$ the integrand in \eqref{eq:soft-partition} tends almost everywhere to $\ind_{Q(a)}e^{-\varepsilon|x|^2/2-\xi\cdot x}$. This convergence passes through the minimization. Indeed, \eqref{eq:polytope-balls} bounds $Z_{q,\varepsilon}(a,0)$ uniformly on compact sets of $a$, for $q\ge q_0>1$ and $0\le\varepsilon\le1$. Formula \eqref{eq:soft-coercivity} bounds the minimizing tilts in a common compact set. For bounded $\xi$, an integrable majorant outside $RB$ is
\[
 \exp\big(M|x|-(|x|/R)^{q_0}\big),
\]
with a constant majorant inside. If $a_j\to a$ and $\xi_j\to\xi$, the same almost-everywhere limit holds off $\partial Q(a)$. Dominated convergence and compactness give locally uniform convergence of the integrals, their minima, and their unique minimizers. Letting next $\varepsilon\downarrow0$ yields
\begin{equation}\label{eq:soft-limit}
 H_{q,\varepsilon}(a)\longrightarrow\cP(Q(a))
 \quad\hbox{locally uniformly, first as }q\to\infty\hbox{ and then as }\varepsilon\downarrow0.
\end{equation}
Thus $a\mapsto\cP(Q(a))$ is concave.

Finally choose increasing finite sets $D_m\subset\Sn$, containing $\{\pm e_1,\ldots,\pm e_n\}$ and with dense union. For $a\in C(\Sn)$ put
\[
 Q_m(a)=\{x:\langle u,x\rangle\le e^{a(u)},\ u\in D_m\}.
\]
These bodies decrease to $W[e^a]$ and stay between fixed balls when $a$ ranges over a compact set in the uniform norm. Lemma~\ref{lem:tilt}, or the same compactness argument for minimizing tilts, gives
\[
 \cP(Q_m(a))\downarrow\cP(W[e^a])=\cH(a).
\]
Each function on the left is concave in $a$. Its limit is therefore concave on $C(\Sn)$.

For centered $K,L$, set $a_t=(1-t)\log h_K+t\log h_L$. The established concavity and \eqref{eq:tilt-volume} give
\begin{equation}\label{eq:volume-chain}
 \log|W[h_K^{1-t}h_L^t]|
 \ge\cH(a_t)\ge(1-t)\log|K|+t\log|L|.
\end{equation}
This proves the inequality in Theorem~\ref{thm:logbm}. The intermediate Wulff body has not been assumed centered. 

\end{proof}

\subsection{Equality and rigidity}\label{subsec:rigidity}
\begin{proof}[Proof of the equality statement in Theorem~\ref{thm:logbm}]\ 
	
Suppose equality holds in \eqref{eq:main-bm} at $t_0\in(0,1)$. Separate positive dilations reduce to $|K|=|L|=1$. The chain \eqref{eq:volume-chain} forces $\cH(a_{t_0})=0$, where $a_t=(1-t)\log h_K+t\log h_L$. Concavity and the zero endpoint values then give
\begin{equation}\label{eq:affine-hypothesis}
 \cH(a_t)=0\qquad(0\le t\le1).
\end{equation}
We show that \eqref{eq:affine-hypothesis} alone, for centered unit-volume endpoints, implies \eqref{eq:main-split} with $\prod_i a_i^{\dim E_i}=1$. In particular, none of the following steps assumes that an intermediate body is centered or has volume one.

Set $f=\log(h_L/h_K)$ and $C_t=W[e^{a_t}]$. We first obtain
\begin{equation}\label{eq:common-measure}
 V_K=V_L=\omega_{C_t},\qquad
 \int_{\Sn} f\,dV_K=0\qquad(0<t<1),
\end{equation}
together with a common support for the three surface area measures and
\begin{equation*}
 \log h_{C_t}=a_t\quad\text{on }\supp V_K.
\end{equation*}
The derivative at $a_0$ is the positive continuous linear functional
\[
 \ell(z)=D\cH(\log h_K)[z]=n\int z\,dV_K.
\]
By concavity, it is a global supergradient of $\cH$. The constant value along the segment gives $\ell(f)=0$, so the same supporting affine functional touches $\cH$ at every $a_t$, including $a_1$. Put $k_t=\log h_{C_t}$. Then $k_t\le a_t$ and $\cH(k_t)=\cH(a_t)=0$. The supergradient inequality and positivity give
\[
 0\le\ell(k_t-a_t)\le0.
\]
Thus $\ell$ also supports $\cH$ at $k_t$. The two-sided derivative in Lemma~\ref{lem:H-variation} makes this supergradient unique. Applying that formula at $k_t$ and $a_1$ proves \eqref{eq:common-measure}. The weights in \eqref{eq:omega} are positive, so the normal supports agree. Finally $a_t-k_t$ is continuous, nonnegative, and has zero $V_K$-integral, which proves the contact assertion.

Use the direction sets $D_m$ constructed in Section~\ref{subsec:concavity} and put $Q_m(t)=Q_m(a_t)$. Common inner and outer balls $rB,RB$ contain all these bodies. Moreover,
\[
 e^{-\|f\|_\infty|s-t|}Q_m(t)\subset Q_m(s)
 \subset e^{\|f\|_\infty|s-t|}Q_m(t).
\]
The continuous functions $t\mapsto\cP(Q_m(t))$ decrease to zero. Dini's theorem makes the convergence uniform. Fix $0<\varepsilon_j\downarrow0$. Choose increasing $m_j$ so that
\[
 0\le\cP(Q_{m_j}(t))\le\varepsilon_j\qquad(0\le t\le1).
\]
On $Q_{m_j}(t)\subset RB$, inserting the factor $e^{-\varepsilon_j|x|^2/2}$ changes the optimized logarithmic integral by at most $\varepsilon_jR^2/2$. After $m_j$ has been fixed, choose even $q_j\to\infty$ large enough that the soft-constraint error is uniformly at most $\varepsilon_j$, and also
\begin{equation}\label{eq:choice-q}
 q_j\ge2n+20,\qquad q_j\ge j(1+\log(N_j+1)),\qquad N_j=|D_{m_j}|.
\end{equation}
By the locally uniform limit in \eqref{eq:soft-limit}, these choices are possible. The smooth concave functions
\[
 G_j(t)=H_{q_j,\varepsilon_j}\big((a_t(u))_{u\in D_{m_j}}\big)
\]
satisfy $\sup_{[0,1]}|G_j|\le C\varepsilon_j$. Secant bounds on the middle interval give
\[
 \int_{1/3}^{2/3}-G_j''(t)\,dt
 =G_j'(1/3)-G_j'(2/3)\le C_1\varepsilon_j.
\]
Choose $t_j\in[1/3,2/3]$ with $-G_j''(t_j)\le C_2\varepsilon_j$. The same secant bounds give $|G_j'(t_j)|\le C_3\varepsilon_j$. Passing to a subsequence, let $t_j\to t_*\in[1/3,2/3]$ and put $C=C_{t_*}$.

Write $D_{m_j}=\{p_{1,j},\ldots,p_{N_j,j}\}$ and, at $t_j$, define
\begin{equation}\label{eq:equality-regularization}
 \begin{gathered}
 \psi_{i,j}(x)=\left(\frac{\langle p_{i,j},x\rangle_+}{e^{a_{t_j}(p_{i,j})}}\right)^{q_j},\qquad
 V_j(x)=\frac{\varepsilon_j}2|x|^2+\sum_i\psi_{i,j}(x)+\xi_j\cdot x,\\
 Z_j=\int e^{-V_j}=e^{G_j(t_j)},\qquad d\mu_j=Z_j^{-1}e^{-V_j}\,dx,
 \end{gathered}
\end{equation}
where $\xi_j$ is the minimizing tilt. These probabilities are centered.

The common ball estimates, \eqref{eq:soft-coercivity}, and \eqref{eq:choice-q} bound the tilts uniformly. On every compact subset of $\intt C$, the normalized support ratios in \eqref{eq:equality-regularization} are at most some fixed $\rho<1$ for large $j$. For each fixed $k$, the derivatives of their sum are bounded by $C_kN_jq_j^k\rho^{q_j-k}\to0$. Outside $C$, a fixed direction from the dense union strictly violates the limiting support constraint; its corresponding power tends to infinity. Common coordinate constraints give an integrable majorant of the form $\exp(M|x|-(|x|/R_0)^{q_0})$ off a fixed ball, with $q_0>1$. Dominated convergence, including every fixed polynomial weight, therefore shows that any subsequential tilt limit centers $\ind_Ce^{-\xi\cdot x}$. Lemma~\ref{lem:tilt} identifies it as $\xi_C$. Consequently,
\begin{equation}\label{eq:equality-target-limit}
 \xi_j\to\xi:=\xi_C,\qquad Z_j\to1,\qquad
 d\mu_j\to d\mu:=\ind_Ce^{-\xi\cdot x}\,dx.
\end{equation}
The same dominated convergence with weight $1+|x|$ gives
\[
 \int_{\R^n}(1+|x|)\big|Z_j^{-1}e^{-V_j(x)}-\ind_C(x)e^{-\xi\cdot x}\big|\,dx\longrightarrow0.
\]
In particular, the density convergence is in $L^1$ and $W_1(\mu_j,\mu)\to0$: couple the common part of the densities identically and bound the transport cost of the remaining equal masses through the origin by the $|x|$-weighted difference. All fixed polynomial moments converge, and $V_j\to\xi\cdot x$ locally in every fixed differentiability class on $\intt C$. The limiting measure is a centered probability, although the uniform measure on $C$ need not yet be centered.

Set
\[
 F_j=-q_j\sum_i f(p_{i,j})\psi_{i,j},\quad
 \Sigma_j=\int xx^T\,d\mu_j,\quad
 \beta_j=\Sigma_j^{-1}\int xF_j\,d\mu_j,\quad
 U_j=F_j-\int F_j\,d\mu_j-\beta_j\cdot x.
\]
The Schur-complement formula and the parameter choice give
\begin{equation*}
 D_j:=q_j^2\sum_i f(p_{i,j})^2\int\psi_{i,j}\,d\mu_j
       -\int U_j^2\,d\mu_j
 =-G_j''(t_j)\le C_2\varepsilon_j,
\end{equation*}
and
\begin{equation}\label{eq:F-mean}
 \left|\int F_j\,d\mu_j\right|=|G_j'(t_j)|\le C_3\varepsilon_j.
\end{equation}
By Lemma~\ref{lem:stability}, the moment potentials may be translated so that $\varphi_j\to\varphi$ locally in every fixed $C^k$, where $\varphi$ is the moment potential of $\mu$. Put $x_j=\nabla\varphi_j$ and $\tau_j=D^2\varphi_j$. On each source compact set, $\tau_j$ and $\tau_j^{-1}$ are uniformly bounded, and $e^{-\varphi_j}$ is bounded above and below by positive constants.

The approximation statement \eqref{eq:retained-fields} gives tests $v_j\in C_c^\infty$ whose fields satisfy
\[
 \|u_j-U_j\|_{L^2(\mu_j)}\le\varepsilon_j,\qquad
 \mathcal T(v_j)+\varepsilon_j\int|W_j|^2\,d\mu_j
 \le D_j+\varepsilon_j^2\le C_4\varepsilon_j.
\]
Thus
\begin{equation}\label{eq:uniform-flux}
 \mathcal T(v_j)\to0,\qquad \sup_j\int|W_j|^2\,d\mu_j<\infty.
\end{equation}
The retained quadratic regularizer is essential for the second bound. In source coordinates, \eqref{eq:derivative-remainder} and \eqref{eq:derivative-comparison} imply
\begin{equation*}
 \int_{B_R}|D_zP_j|^2\,dz\longrightarrow0\qquad(R<\infty).
\end{equation*}
Let $A_j$ be the average of $P_j$ over $B_1$. Poincar\'e's inequality and the identity
\begin{equation*}
 W_j=P_jx_j-\diver_zP_j
\end{equation*}
give $P_j-A_j\to0$ and $W_j-A_jx_j\to0$ in $L^2(B_1)$. The global bound in \eqref{eq:uniform-flux} bounds $W_j$ there, since the source density is bounded below. Also
\[
 \int_{B_1}x_jx_j^T\,dz\longrightarrow
 \int_{B_1}(\nabla\varphi)(\nabla\varphi)^T\,dz\succ0.
\]
The last matrix is positive definite because the gradient image of $B_1$ is open. Therefore $A_j$ is bounded. After a subsequence, $A_j\to A$. Applying Poincar\'e's inequality on larger balls and comparing their averages on $B_1$, we obtain
\begin{equation*}
 P_j\to A,\qquad W_j\to A\nabla\varphi
 \quad\hbox{in }L^2_\loc(\R^n_z).
\end{equation*}
The convergent moment maps and their inverses give $W_j(x)\to Ax$ in $L^2_\loc(\intt C)$. Since $B_j=P_j\tau_j$ is symmetric, passage to the limit gives
\begin{equation}\label{eq:matrix-commutation}
 AD^2\varphi(z)=D^2\varphi(z)A^T\qquad(z\in\R^n).
\end{equation}

Define a positive measure on positions and directions by
\begin{equation}\label{eq:joint-layer}
 d\Lambda_j(x,u)=\sum_i q_j\psi_{i,j}(x)\,d\mu_j(x)\,\delta_{p_{i,j}}(du).
\end{equation}
The direction is $p_{i,j}$ itself, because the powers are one-sided. Radial integration by parts gives
\begin{align*}
 \Lambda_j(\R^n\times\Sn)&=n-\varepsilon_j\int|x|^2\,d\mu_j,
 \\
 \int|x|^2\,d\Lambda_j
 &=(n+2)\int|x|^2\,d\mu_j-\varepsilon_j\int|x|^4\,d\mu_j
   -\int|x|^2\xi_j\cdot x\,d\mu_j.
\end{align*}
The common moment bounds imply bounded mass and second moment, hence tightness and uniform integrability of the first moments. For a compact smooth function $g$,
\begin{align*}
 \int g(x)\,d\Lambda_j(x,u)
 &=\int\big(ng+x\cdot\nabla g-\varepsilon_jg|x|^2-g\xi_j\cdot x\big)\,d\mu_j\\
 &\longrightarrow\int_C(ng+x\cdot\nabla g-g\xi\cdot x)e^{-\xi\cdot x}\,dx\\
 &=\int_{\partial C}g(x)h_C(\nu_C)e^{-\xi\cdot x}\,d\mathcal H^{n-1}(x).
\end{align*}
Thus the limiting position marginal lies on $\partial C$.

Every joint limit is supported on contact pairs
\begin{equation*}
 x\in\partial C,\qquad \langle x,u\rangle=h_C(u)=e^{a_{t_*}(u)}.
\end{equation*}
Indeed, on a compact neighborhood where $\langle x,u\rangle\le(1-\delta)e^{a_{t_*}(u)}$, the mass in \eqref{eq:joint-layer} is at most $N_jq_j(1-\delta/2)^{q_j}$ for large $j$, and tends to zero by \eqref{eq:choice-q}. On the limiting boundary one always has $\langle x,u\rangle\le h_C(u)\le e^{a_{t_*}(u)}$. This proves the asserted restriction to contact pairs. The outer normal is unique at almost every boundary point, and the position marginal gives no mass to the exceptional set. Hence the joint limit is unique:
\begin{equation*}
 d\Lambda_j(x,u)\rightharpoonup
 h_C(\nu_C(x))e^{-\xi\cdot x}\,d\mathcal H^{n-1}(x)\,
 \delta_{\nu_C(x)}(du).
\end{equation*}
Since $F_j\,d\mu_j=-\int f(u)\,d\Lambda_j(x,u)$ and $f$ is continuous,
\[
 F_j\,d\mu_j\rightharpoonup
 -f(\nu_C)h_C(\nu_C)e^{-\xi\cdot x}\,d\mathcal H^{n-1}|_{\partial C}.
\]
The first-moment control permits linear-growth tests, so $\int xF_j\,d\mu_j$ converges. Because $\Sigma_j\to\Sigma=\int_Cxx^T\,d\mu\succ0$, the vectors $\beta_j$ converge to some $\beta\in\R^n$. The approximation of $U_j$ and \eqref{eq:F-mean} now give
\begin{equation}\label{eq:u-limit}
 u_j\,d\mu_j\rightharpoonup
 -f(\nu_C)h_C(\nu_C)e^{-\xi\cdot x}\,d\mathcal H^{n-1}|_{\partial C}
 -\beta\cdot x\,d\mu(x).
\end{equation}
The interior affine term is retained; it need not vanish separately.

For $g\in C_c^\infty(\R^n)$, adjunction gives
\[
 \int u_jg\,d\mu_j=\int W_j\cdot\nabla g\,d\mu_j.
\]
The right side tends to $\int_CAx\cdot\nabla g\,d\mu$. To justify passage across the boundary, first restrict to $C_\delta=\{x\in C:\operatorname{dist}(x,\partial C)\ge\delta\}$, where local convergence applies. The absolute value of the remaining integral is at most
\[
 \|\nabla g\|_\infty
 \left(\int|W_j|^2\,d\mu_j\right)^{1/2}
 \mu_j(\R^n\setminus C_\delta)^{1/2}.
\]
Its limit superior tends to zero as $\delta\downarrow0$, by \eqref{eq:uniform-flux} and the $L^1$ convergence in \eqref{eq:equality-target-limit}. Hence \eqref{eq:u-limit} implies
\begin{equation*}
 \int_C Ax\cdot\nabla g\,d\mu
 =-\int_{\partial C}gf(\nu_C)h_C(\nu_C)e^{-\xi\cdot x}\,d\mathcal H^{n-1}
   -\int_Cg\beta\cdot x\,d\mu.
\end{equation*}
On the other hand, weighted integration by parts in $C$ gives
\[
 \int_C Ax\cdot\nabla g\,d\mu
 =\int_{\partial C}gAx\cdot\nu_Ce^{-\xi\cdot x}\,d\mathcal H^{n-1}
  +\int_Cg\big((A^T\xi)\cdot x-\tr A\big)\,d\mu.
\]
Testing first inside $C$ yields an affine identity on an open set, and therefore
\begin{equation}\label{eq:affine-flux}
 \tr A=0,\qquad \beta=-A^T\xi.
\end{equation}
Subtracting the interior terms proves
\begin{equation}\label{eq:boundary-flux}
 Ax\cdot\nu_C(x)=-f(\nu_C(x))h_C(\nu_C(x))
 \quad\hbox{for }\mathcal H^{n-1}\hbox{-almost every }x\in\partial C.
\end{equation}
Here restrictions of compact smooth tests are dense in $C(\partial C)$, so the weak identity is an identity of finite boundary measures. No conclusion that $\xi=0$ has been used.

Equation \eqref{eq:matrix-commutation} makes $A$ self-adjoint for a positive definite inner product. For example, if $\tau_0=D^2\varphi(0)$, then $\tau_0^{-1/2}A\tau_0^{1/2}$ is symmetric. Use the constant dual change
\[
 z=Tz',\quad x'=T^Tx,\qquad
 T=(\det\tau_0)^{1/(2n)}\tau_0^{-1/2},\qquad\det T=1.
\]
It makes the transformed $A$ symmetric. The common normal support and the flux identity are preserved: the flux and the supporting distance transform with the same normalizing factor, as in \eqref{eq:covariance}. We henceforth work in these coordinates and omit primes.

Let $\lambda_1,\ldots,\lambda_r$ be the distinct eigenvalues of $A$, with orthogonal eigenspaces $E_i$. Equation \eqref{eq:matrix-commutation} says that $D^2\varphi$ commutes with $A$. All mixed Hessian blocks between distinct eigenspaces vanish. Since $\varphi$ is defined on the whole product space, integration gives
\[
 \varphi(z_1+\cdots+z_r)=c+\sum_i\varphi_i(z_i),\qquad z_i\in E_i.
\]
Its gradient image is a Cartesian product. Taking closures, we obtain
\[
 C=C_1\oplus\cdots\oplus C_r,\qquad C_i\subset E_i.
\]
Each factor has nonempty relative interior and contains zero in that interior. The boundary of a product, or Lemma~\ref{lem:product}, gives
\[
 \supp S_C\subset\bigcup_i S(E_i).
\]
Let $S$ be the common support of $S_C,S_K,S_L$ in \eqref{eq:common-measure}, and put $S_i=S\cap E_i$. Reconstruction in Lemma~\ref{lem:reconstruct} separates the constraints and yields
\[
 K=\bigoplus_iK_i,\qquad L=\bigoplus_iL_i,
\]
where
\[
 K_i=\{y\in E_i:\langle y,u\rangle\le h_K(u),\ u\in S_i\},\qquad
 L_i=\{y\in E_i:\langle y,u\rangle\le h_L(u),\ u\in S_i\}.
\]
These factors are bounded and full-dimensional relative to $E_i$, since their products are the compact full-dimensional endpoint bodies. Fubini's theorem gives $g(K_i)=g(L_i)=0$, because the endpoint centroids vanish.

At almost every boundary point of $C$ with normal $u\in E_i$, one has $Ax\cdot u=\lambda_i h_C(u)$. Thus \eqref{eq:boundary-flux} gives $f(u)=-\lambda_i$ for $S_C$-almost every such direction. Continuity extends this equality to $S_i$. It follows that
\[
 h_L(u)=e^{-\lambda_i}h_K(u)\quad(u\in S_i),\qquad
 L_i=e^{-\lambda_i}K_i.
\]
Set $a_i=e^{-\lambda_i}$. By \eqref{eq:affine-flux},
\[
 \prod_i a_i^{\dim E_i}=\exp\left(-\sum_i\lambda_i\dim E_i\right)
 =e^{-\tr A}=1.
\]
Undoing the common linear transformation gives complementary subspaces without an orthogonality requirement. Thus \eqref{eq:affine-hypothesis} implies the claimed common direct-sum representation. In particular, it proves necessity for the normalized equality pair in Theorem~\ref{thm:logbm}. Restoring the original endpoint volumes gives arbitrary positive dilation factors.

Conversely, assume \eqref{eq:main-split}. By \eqref{eq:covariance} we may make the subspaces orthogonal. For $u=\sum_i u_i$, $u_i\in E_i$, put $b_i=h_{K_i}(u_i)\ge0$. H\"older's inequality gives
\[
 \sum_i a_i^s b_i\le\left(\sum_i b_i\right)^{1-s}
                        \left(\sum_i a_i b_i\right)^s,
 \qquad 0<s<1.
\]
Hence $\bigoplus_i a_i^sK_i\subset W[h_K^{1-s}h_L^s]$. The defining inequalities in directions confined to $E_i$ give the reverse inclusion. Taking volumes,
\[
 \big|W[h_K^{1-s}h_L^s]\big|=|K|\prod_i a_i^{s\dim E_i}
 =|K|^{1-s}|L|^s.
\]
Thus the direct-sum condition is sufficient, and the proof is complete.
\end{proof}

\section{Proofs of Theorems 1.2 and 1.3}\label{sec:proof-m-uniqueness}\ 

The first variation gives the logarithmic Minkowski inequality from the endpoint volume inequality. Its equality case uses the implication from \eqref{eq:affine-hypothesis} established in Section~\ref{subsec:rigidity}, rather than a concavity assertion for the ordinary volumes of intermediate, possibly noncentered bodies.

\begin{proof}[Proof of Theorem~\ref{thm:logm}]
Let $K,L\in\cK$ and set
\[
 f=\log(h_L/h_K),\qquad a_t=\log h_K+tf,\qquad C_t=W[e^{a_t}].
\]
For $t>0$, Theorem~\ref{thm:logbm} gives
\[
 \frac{\log|C_t|-\log|K|}{t}\ge\log\frac{|L|}{|K|}.
\]
Aleksandrov's formula \eqref{eq:aleksandrov} at $C_0=K$ yields
\[
 \left.\frac{d}{dt}\right|_{t=0}\log|C_t|
 =\frac n{|K|}\int_{\Sn}f\,dV_K.
\]
Letting $t\downarrow0$ proves \eqref{eq:main-m}. Only the two original endpoints are used in the application of Theorem~\ref{thm:logbm}; no intermediate centroid condition is required.

Suppose equality holds. Put $\Phi(t)=\cH(a_t)$. Concavity \eqref{eq:H-concavity}, Lemma~\ref{lem:tilt}, and the variation formula \eqref{eq:H-variation} give
\[
 \Phi(0)=\log|K|,\qquad \Phi(1)=\log|L|,
 \qquad \Phi'(0)=\frac n{|K|}\int f\,dV_K.
\]
Equality in \eqref{eq:main-m} says $\Phi'(0)=\Phi(1)-\Phi(0)$. The chord and tangent inequalities for the concave function $\Phi$ imply
\[
 (1-t)\Phi(0)+t\Phi(1)\le\Phi(t)\le\Phi(0)+t\Phi'(0).
\]
The two bounds coincide, so
\begin{equation*}
 \cH(a_t)=(1-t)\log|K|+t\log|L|\qquad(0\le t\le1).
\end{equation*}
Normalize the endpoints by $K_0=|K|^{-1/n}K$ and $L_0=|L|^{-1/n}L$. Their logarithmic support segment is
\[
 a_t^0=a_t-\frac{1-t}{n}\log|K|-\frac t n\log|L|.
\]
The scaling identity \eqref{eq:H-basic} and the preceding affine equality give $\cH(a_t^0)=0$. The affine-segment argument in Section~\ref{subsec:rigidity} therefore applies to $K_0,L_0$. Restoring their scales gives the representation \eqref{eq:main-split} for $K,L$.

Conversely, suppose \eqref{eq:main-split} holds and put $d_i=\dim E_i$. The volume identity established in the proof of Theorem~\ref{thm:logbm} gives
\[
 |C_t|=|K|\prod_i a_i^{td_i},\qquad
 \frac{|L|}{|K|}=\prod_i a_i^{d_i}.
\]
Differentiate the first identity at zero and use \eqref{eq:aleksandrov}. Then
\[
 \int_{\Sn}\log\frac{h_L}{h_K}\,dV_K
 =\frac{|K|}{n}\sum_i d_i\log a_i
 =\frac{|K|}{n}\log\frac{|L|}{|K|}.
\]
Thus the stated condition is also sufficient.
\end{proof}

\begin{proof}[Proof of Theorem~\ref{thm:uniqueness}]
Suppose first that $\vbar_K=\vbar_L=\mu$. Divide the logarithmic Minkowski inequality by the volume of its first body, and apply it in both orders. We obtain
\[
 \int_{\Sn}\log\frac{h_L}{h_K}\,d\mu
 \ge\frac1n\log\frac{|L|}{|K|},\qquad
 \int_{\Sn}\log\frac{h_K}{h_L}\,d\mu
 \ge\frac1n\log\frac{|K|}{|L|}.
\]
The two deficits are nonnegative and sum to zero. Both inequalities are therefore equalities, and Theorem~\ref{thm:logm} gives \eqref{eq:main-split}. Conversely, every pair with this representation has equal normalized cone-volume measures by Lemma~\ref{lem:product}.
\end{proof}

In particular, a cone-volume measure with full support on $\Sn$ has at most one centered realization: a nontrivial direct sum would place its support in a finite union of proper normal subspaces by \eqref{eq:normal-subspaces}. This consequence concerns uniqueness of an existing centered realization and makes no assertion about the existence of one.

\section{Proofs of Theorems 1.4 and 1.5}\label{sec:lp-proofs}\ 

Throughout this section, $K,L\in\mathcal K^n_o$ unless centering is explicitly imposed. For $a,b>0$ write
\[
 M_{p,t}(a,b)=((1-t)a^p+tb^p)^{1/p}\quad(p\ne0),
 \qquad M_{0,t}(a,b)=a^{1-t}b^t.
\]
Thus $C_{p,t}(K,L)=W[M_{p,t}(h_K,h_L)]$. We use the scalar normalization argument from the $L_p$ theory \cite{BLYZ12,KM22}, now for the optimized functional $\Psi$. Strictness will follow from Lemma~\ref{lem:reconstruct} and the strict monotonicity in Lemma~\ref{lem:tilt}; no logarithmic equality classification is needed.

\begin{proof}[Proof of Theorem~\ref{thm:lp-bm}]
The endpoint cases $t=0,1$ are immediate. Assume $0<t<1$. Set
\[
 A=\Psi(K)^{1/n},\qquad B=\Psi(L)^{1/n},\qquad
 r=\big((1-t)A^p+tB^p\big)^{1/p},\qquad
 \lambda=\frac{tB^p}{r^p}.
\]
Then $0<\lambda<1$ and $1-\lambda=(1-t)A^p/r^p$.  Put
$K_0=A^{-1}K$ and $L_0=B^{-1}L$, so that
$\Psi(K_0)=\Psi(L_0)=1$.  The scalar weighted arithmetic--geometric
mean inequality gives, pointwise on the sphere,
\begin{equation*}
 M_{p,t}(h_K,h_L)
   =r\,M_{p,\lambda}(h_{K_0},h_{L_0})
   \geq r\,h_{K_0}^{1-\lambda}h_{L_0}^{\lambda}.
\end{equation*}
Consequently, with
\[
 D=W[h_{K_0}^{1-\lambda}h_{L_0}^{\lambda}],\qquad
 Q=r^{-1}C_{p,t}(K,L)
   =W[M_{p,\lambda}(h_{K_0},h_{L_0})],
\]
one has $D\subset Q$.  Concavity of $\cH$ gives
$\Psi(D)\geq1$, and monotonicity gives $\Psi(Q)\geq1$.
Homogeneity now proves \eqref{eq:lp-weighted-BM}.

Suppose equality holds.  Then $\Psi(Q)=1$, and the preceding chain
forces $\Psi(D)=1$.  Strict monotonicity implies $D=Q$.  Apply
Lemma~\ref{lem:reconstruct} to the two descriptions of this same
Wulff body.  On $\supp S_Q$,
\[
 h_{K_0}^{1-\lambda}h_{L_0}^{\lambda}
       =h_Q=M_{p,\lambda}(h_{K_0},h_{L_0}).
\]
Equality in the scalar arithmetic--geometric mean inequality, with
$p>0$ and $0<\lambda<1$, implies
\[
 h_{K_0}=h_{L_0}=h_Q
                    \quad\text{on }\supp S_Q.
\]
The reconstruction formula \eqref{eq:reconstruct},
applied to $Q$, gives $K_0,L_0\subset Q$.  All three bodies have
$\Psi$-value one, so strict monotonicity yields $K_0=L_0=Q$.
Thus $L=(B/A)K$.

Conversely, if $L=cK$, then
$C_{p,t}(K,L)=M_{p,t}(1,c)K$, and homogeneity verifies equality.

For centered endpoints, \eqref{eq:tilt-volume} and the optimized inequality just proved give
\[
 |C_{p,t}(K,L)|^{p/n}
 \geq\Psi(C_{p,t}(K,L))^{p/n}
 \geq(1-t)|K|^{p/n}+t|L|^{p/n}.
\]
If \eqref{eq:lp-centered-BM} is an equality, equality holds throughout
this chain, and the optimized equality characterization gives $L=cK$.
Conversely, such a dilation gives equality directly.

Finally, the scalar arithmetic--geometric mean inequality gives
\[
 \big((1-t)|K|^{p/n}+t|L|^{p/n}\big)^{n/p}
             \geq |K|^{1-t}|L|^t,
\]
with equality precisely when $|K|=|L|$.  Equality in
\eqref{eq:lp-centered-geometric} therefore requires both $L=cK$ and
$|K|=|L|$, hence $c=1$.  The converse is immediate.
\end{proof}

The same normalization works for arbitrary positive continuous constraint functions, not only support functions. Indeed, for $a,b>0$ on $\Sn$, normalize the bodies $W[a^{1/p}]$ and $W[b^{1/p}]$ by their $\Psi$-values and apply \eqref{eq:H-concavity} directly to the logarithms of the normalized constraint functions. The pointwise arithmetic--geometric mean inequality then gives
\begin{equation*}
 a\longmapsto\Psi(W[a^{1/p}])^{p/n}
 \quad\text{concave on }\{a\in C(\Sn):a>0\}.
\end{equation*}
This is a statement about functional values; the equality assertion in Theorem~\ref{thm:lp-bm} concerns bodies represented by their actual support functions.

\begin{proof}[Proof of Theorem~\ref{thm:lp-minkowski}]
Concavity of $\cH$ and its first variation
\eqref{eq:H-variation}, applied at $\log h_K$ in the
direction $\log(h_L/h_K)$, give
\begin{equation}\label{eq:lp-weighted-log-minkowski}
 \int\log\frac{h_L}{h_K}\,d\omega_K
      \geq\frac1n\log\frac{\Psi(L)}{\Psi(K)}.
\end{equation}
Since $\omega_K$ is a probability measure, Jensen's inequality yields
\begin{align*}
 \int\left(\frac{h_L}{h_K}\right)^p\,d\omega_K
 &\geq\exp\left(p\int\log\frac{h_L}{h_K}\,d\omega_K\right)
 \geq\left(\frac{\Psi(L)}{\Psi(K)}\right)^{p/n}.
\end{align*}
Multiplying by $\Psi(K)$ proves \eqref{eq:lp-weighted-minkowski}.

If equality holds, the two preceding inequalities are
equalities.  Strict convexity of $s\mapsto e^{ps}$ implies that
$h_L/h_K$ is constant $\omega_K$-almost everywhere.  Its value must be
\[
 c=\left(\frac{\Psi(L)}{\Psi(K)}\right)^{1/n}.
\]
Continuity and equality of the supports of $\omega_K$ and $S_K$ give
$h_L=c h_K$ on $\supp S_K$.  Hence
$L\subset cK$ by \eqref{eq:lp-reconstruction-consequence}.  Homogeneity
gives $\Psi(cK)=\Psi(L)$, and strict monotonicity in
Lemma~\ref{lem:tilt} forces $L=cK$.

Conversely, if $L=cK$, then
$\widetilde V_p(K,L)=c^p\Psi(K)$, which is exactly the right-hand side
of \eqref{eq:lp-weighted-minkowski}.

For centered $K,L$, the minimizing tilts vanish. Consequently $\Psi(K)=|K|$, $\Psi(L)=|L|$, and $\widetilde V_p(K,L)=V_p(K,L)$. This proves \eqref{eq:lp-centered-minkowski}, with the same equality characterization.

It remains to prove the assertion about $L_p$ surface area measures. Suppose the two measures agree.  By \eqref{eq:lp-standard},
\[
 V_p(K,L)=\frac1n\int h_L^p\,dS_p(K)
         =\frac1n\int h_L\,dS_L=|L|,
 \qquad V_p(L,K)=|K|.
\]
Writing $\theta=1-p/n$, the two applications of
\eqref{eq:lp-centered-minkowski} give
\[
 |L|^\theta\geq |K|^\theta,
 \qquad |K|^\theta\geq |L|^\theta.
\]
If $p\neq n$, then $\theta\neq0$ and hence $|K|=|L|$.
Equality holds in \eqref{eq:lp-centered-minkowski}, so $L=cK$;
equality of volume gives $c=1$.

If $p=n$, the identity $V_n(K,L)=|L|$ is already equality in
\eqref{eq:lp-centered-minkowski}, and therefore $L=cK$.
Conversely, scaling surface area and support functions gives
\[
 S_p(cK,\cdot)=c^{n-p}S_p(K,\cdot).
\]
Thus every positive dilation preserves $S_n$, as asserted.  The converse
when $p\neq n$ and $K=L$ is immediate.
\end{proof}

As $p\downarrow0$, the logarithms of the constraint functions converge
uniformly:
\[
 \log M_{p,t}(h_K,h_L)
       \longrightarrow(1-t)\log h_K+t\log h_L.
\]
The $n$-Lipschitz estimate for $\cH$ therefore passes the limit
through the left-hand side of the logarithmic form of
\eqref{eq:lp-weighted-BM}.  Since
\[
 \frac np\log\big((1-t)e^{p\cP(K)/n}
                         +te^{p\cP(L)/n}\big)
       \longrightarrow(1-t)\cP(K)+t\cP(L),
\]
the limiting statement is exactly
\begin{equation*}
 \Psi(W[h_K^{1-t}h_L^t])
       \geq\Psi(K)^{1-t}\Psi(L)^t.
\end{equation*}

For the mixed-volume inequality the nontrivial limit is the first-order
term, rather than the value at $p=0$.  Indeed,
\begin{align*}
 \frac{\widetilde V_p(K,L)}{\Psi(K)}
   &=1+p\int\log\frac{h_L}{h_K}\,d\omega_K+O(p^2),\\
 \left(\frac{\Psi(L)}{\Psi(K)}\right)^{p/n}
   &=1+\frac pn\log\frac{\Psi(L)}{\Psi(K)}+O(p^2).
\end{align*}
The logarithm of the support-function ratio is bounded, so these
expansions follow directly under the integral.  Subtracting one and
dividing by $p>0$ recovers
\eqref{eq:lp-weighted-log-minkowski}.  By contrast,
$\widetilde V_0(K,L)=\Psi(K)$ carries no information about $L$.
For centered bodies the first-order limit is the logarithmic Minkowski
inequality
\begin{equation*}
 \int\log\frac{h_L}{h_K}\,d\vbar_K
                  \geq\frac1n\log\frac{|L|}{|K|}.
\end{equation*}

The dilation-only equality statements for $p>0$ do not persist at $p=0$: strictness of the scalar power-mean comparison disappears. The equality cases in Theorem~\ref{thm:logbm} and the product formula \eqref{eq:product-normalized} give logarithmic equality for independent dilations of common direct summands, even when the dilation factors differ. This is precisely the additional freedom classified in Theorems~\ref{thm:logbm}--\ref{thm:uniqueness}.

For completeness, the following negative-exponent examples show why the universal volume and mixed-volume extensions stop at $p=0$. The rectangular-box obstruction is also recorded in \cite[Section 1]{KM22}. Fix $p<0$. The proposed volume extension would read
\begin{equation}\label{eq:lp-negative-BM-claim}
 |C_{p,t}(K,L)|\ge
 \big((1-t)|K|^{p/n}+t|L|^{p/n}\big)^{n/p}.
\end{equation}
Let $B_\infty^n=[-1,1]^n$ and choose $s\neq0$.  Set
\[
 K_s=\operatorname{diag}(e^s,e^{-s},1,\ldots,1)B_\infty^n,
 \qquad L_s=K_{-s}.
\]
Both volumes equal $2^n$.  For $t=1/2$, put
\[
 r=\left(\frac{e^{ps}+e^{-ps}}2\right)^{1/p}
   =\cosh(ps)^{1/p}<1.
\]
The defining constraints in the coordinate directions alone imply
\[
 C_{p,1/2}(K_s,L_s)
       \subset[-r,r]^2\times[-1,1]^{n-2}.
\]
Its volume is therefore at most $2^n r^2<2^n$, whereas the right-hand
side of \eqref{eq:lp-negative-BM-claim} is $2^n$.  This disproves \eqref{eq:lp-negative-BM-claim}.  Notice that when $p<0$, raising the volume inequality
\eqref{eq:lp-negative-BM-claim} to the power $p/n$ reverses its
inequality sign.

For the mixed-volume lower bound, take
\[
 K=B_\infty^n,\qquad
 L=(n!)^{1/n}B_1^n,
 \qquad B_1^n=\left\{x:\sum_{i=1}^n|x_i|\leq1\right\}.
\]
The volumes again equal $2^n$.  The surface area measure of $K$ is
supported on the $2n$ coordinate directions, each with mass $2^{n-1}$,
and $h_K(\pm e_i)=1$, $h_L(\pm e_i)=(n!)^{1/n}$.  Consequently
\[
 \frac{V_p(K,L)}{|K|}=(n!)^{p/n}<1,
\]
contradicting the proposed lower bound.  In dimension two this is the
square and the equal-area diamond
$\{|x_1|+|x_2|\leq\sqrt2\}$.

For the mixed-volume upper bound, keep $K=B_\infty^n$ and take
$L=K_s$.  These bodies also have equal volume, but
\[
 \frac{V_p(K,L)}{|K|}
       =\frac{e^{ps}+e^{-ps}+n-2}{n}
       =\frac{2\cosh(ps)+n-2}{n}>1.
\]
Thus the upper bound fails as well.  Every body used in the proof is
origin-symmetric and hence centered; its optimizing exponential tilt is
zero.  The same examples therefore also rule out the corresponding
negative-exponent extensions with $\Psi$ and $\widetilde V_p$.

\paragraph{Authors' note.}
Since 2018, the first two authors—then fourth-year undergraduate students—have been studying the logarithmic Brunn–Minkowski inequality and the non-symmetric logarithmic Minkowski problem under the supervision of the third author. Over the subsequent eight years, we continued to think about these problems and obtained some results concerning the existence and uniqueness of solutions to the  non-symmetric logarithmic Minkowski problem.

After OpenAI released a preprint giving a proof of the logarithmic Brunn–Minkowski inequality for origin-symmetric convex bodies, the second author, with the assistance of GPT, obtained a complete characterization of the equality cases together with a proof. We then provided GPT with both the proof of the origin-symmetric logarithmic Brunn–Minkowski inequality and the argument for its equality cases, and asked it to investigate the logarithmic Brunn–Minkowski conjecture for general convex bodies with centroid at the origin. Approximately half an hour later, GPT proposed the core idea underlying the proof developed in the present paper. Through many subsequent rounds of interaction, the arguments leading to the main theorems of the current manuscript gradually took shape.

We used several separate GPT sessions to conduct adversarial checks of the key calculations, limiting arguments, equality cases, and possible logical gaps. These checks do not substitute for independent human verification. The authors have not yet completed an independent  verification of the manuscript, and we therefore present this work as a proposed proof requiring further careful scrutiny. We will continue to examine the manuscript line by line and continue to refine both the mathematical arguments and the exposition.

We do not intend to submit this manuscript to a traditional academic journal. We hope to make this work openly and fully available to the mathematical community, both in the hope that it may provide useful material for future work in the area and to commemorate the problems that accompanied us throughout much of our doctoral research.

Finally, we are deeply grateful to K\'aroly J. B\"or\"oczky, Erwin Lutwak, Deane Yang, and Gaoyong Zhang for formulating the logarithmic Brunn–Minkowski conjecture for origin-symmetric convex bodies, and to Martin Henk for proposing its extension to convex bodies with centroid at the origin. These beautiful and challenging problems have inspired our research for many years.

\end{document}